\documentclass{article}

\usepackage{authblk}

\usepackage{amsthm,amsmath,amssymb}

\usepackage[numbers]{natbib}

\usepackage{thmtools}

\usepackage{hyperref}
\hypersetup{
    hidelinks,
}

\usepackage{hypernat}
\usepackage[margin=1in]{geometry}

\newtheorem{theorem}{Theorem}[section]
\newtheorem{proposition}[theorem]{Proposition}
\newtheorem{lemma}[theorem]{Lemma}
\newtheorem{corollary}[theorem]{Corollary}
\newtheorem{remark}[theorem]{Remark}

\theoremstyle{definition}

\newcommand{\E}{{\mathbb E}}

\newcommand{\R}{{\mathbb R}}
\renewcommand{\P}{{\mathbb P}}

\newcommand{\C}{{\mathcal{C}}}

\newcommand{\T}{{\mathcal{T}}}

\newcommand{\K}{{\mathcal{K}}}

\newcounter{rcnt}[section]

\def\qt#1{\qquad\text{#1}}

\def\argmin{\mathop{\rm argmin}}

\newcommand{\eat}[1]{} 
\newcommand{\op}[1]{\operatorname{#1}}
\usepackage{mathtools}
\mathtoolsset{showonlyrefs}
\usepackage{tikz}
\renewcommand{\S}{{\mathcal{S}}}

\usepackage[font=small]{caption}
\usepackage{mathdots,xcolor,subcaption}
\renewcommand{\b}[1]{\boldsymbol{\mathbf{#1}}}

\usepackage{enumerate}

\DeclarePairedDelimiter{\parens}{(}{)}
\DeclarePairedDelimiter{\braces}{\{}{\}}
\DeclarePairedDelimiter{\brackets}{[}{]}
\DeclarePairedDelimiter{\abs}{|}{|}
\DeclarePairedDelimiter{\norm}{\|}{\|}

\DeclarePairedDelimiter{\inner}{\langle}{\rangle}

\begin{document}

\title{On the risk of convex-constrained least squares estimators
    under misspecification}
\author{Billy Fang \qquad Adityanand Guntuboyina\\Department of Statistics, University of California, Berkeley}
\date{June 13, 2017}

\maketitle

\begin{abstract}
We consider the problem of estimating the
mean of a noisy vector.
When the mean lies in a convex constraint set,
the least squares projection of the random vector
onto the set is a natural estimator.
Properties of the risk of this estimator, such as
its asymptotic behavior as the noise tends to zero, have been well studied.
We instead study the behavior of this estimator under misspecification,
that is, without the assumption that the mean lies in the constraint set.
For appropriately defined notions of risk in the misspecified setting,
we prove
a generalization of a low noise characterization
of the risk due to \citet{oymak2013sharp} in the case of a polyhedral
constraint set.
An interesting consequence of our results is that the risk can be much smaller
in the misspecified setting than in the well-specified setting.
We also discuss consequences of our result for isotonic regression.
\end{abstract}

\section{Introduction}
In many statistical problems, it is common to model the observations
$y_1, \dots, y_n \in \R$ as
$y_i = \theta^*_i + \sigma z_i$
where $\theta_1^*, \dots, \theta_n^*$ are unknown parameters of
interest, $z_1, \dots, z_n$ represent noise or error variables that
have mean zero, and $\sigma > 0$ denotes a scale parameter.
In vector notation, this is
equivalent to writing
\begin{equation}
Y = \theta^* + \sigma Z,
\end{equation}
where $Y \coloneqq (y_1, \dots, y_n)$,
$\theta^* \coloneqq (\theta^*_1, \dots, \theta^*_n)$,
and $Z \coloneqq (z_1,\ldots, z_n)$.
A common instance of this model is the Gaussian sequence model,
where the $z_1, \dots, z_n$ are independent standard Gaussian random
variables, in which case the model can be written as $Y \sim
N(\theta^*, \sigma^2 I_n)$, where $I_n$ is the $n \times n$ identity
matrix.

A standard method of estimating $\theta^*$ from the observation vector
$Y$ is to fix a closed convex set $\C$ of $\R^n$ and use the least
squares estimator under the constraint given by $\theta \in
\C$. Specifically, the least squares projection is
\begin{equation}
    \Pi_{\C}(x) \coloneqq \argmin_{\theta \in \C} \norm{x  - \theta}^2,
\end{equation}
(where $\|\cdot\|$
denotes the standard Euclidean norm in $\R^n$),
and one estimates $\theta^*$ by
\begin{equation}
\hat{\theta}(Y) \coloneqq \Pi_{\C}(Y).
\end{equation}
When $\C$ is taken to be $\{X \beta : \|\beta\|_1 \leq R\}$ for some
deterministic $n \times p$ matrix $X$ and $R > 0$, this estimator
becomes LASSO in the constrained form as originally proposed by
\citet{tibshirani1996regression}. When $\C$ is taken to be $\{X
\beta : \min_j \beta_j \geq 0 \}$, this estimator becomes nonnegative
least squares. Note that shape restricted regression estimators are
special cases of nonnegative least squares for appropriate choices of
$X$ (see, for example, \citet{groeneboom2014nonparametric}). Also, note that
both sets $\{X\beta : \|\beta\|_1 \leq R\}$ and $\{X \beta : \min_j \beta_j
\geq 0\}$ are examples of polyhedral sets. Therefore in most
applications, the constraint set $\C$ is polyhedral.

There exist many results in the literature studying the accuracy of
$\hat{\theta}(Y)$ as an estimator for $\theta^*$. Most of these
results make the assumption that $\theta^* \in \C$. In this paper, we
shall refer to this assumption as the \textit{well-specified}
assumption. Essentially, the constraint set $\C$ can be taken to be a
part of the model specification, and the assumption $\theta^* \in \C$
means that the true mean vector $\theta^*$ satisfies the model
assumptions, i.e. the model is well-specified.

Under the well-specified assumption, it is reasonable and common to
measure the accuracy of $\hat{\theta}(Y)$ via its risk under squared
Euclidean distance. More precisely, the risk of $\hat{\theta}(Y)$ is
defined by
\begin{equation*}
R(\hat{\theta}, \theta^*) \coloneqq  \E_{\theta^*} \|\hat{\theta}(Y) -
\theta^* \|^2
\end{equation*}
where $\E_{\theta^*}$ refers to expectation taken with respect to the
noise $Z$ in the model $Y = \theta^* + \sigma Z$.

Many results on $R(\hat{\theta}, \theta^*)$ in the well-specified
setting  are available in the literature. Of all the available
results, let us isolate two results from \citet{oymak2013sharp}
because of their generality. In the setting where $Z \sim N(0, I_n)$,
\citet{oymak2013sharp} first proved the
upper bound
\begin{equation}\label{ohub}
\frac{1}{\sigma^2} R(\hat{\theta}, \theta^*) \leq \delta(T_{\C}(\theta^*)),
\end{equation}
where $T_\C(\theta^*)$ denotes the \textit{tangent cone} of $\C$ at
$\theta^*$, defined by
\begin{equation}\label{tcd}
  T_{\C}(\theta^*) = \op{cl}\left\{\alpha(\theta - \theta^*) : \alpha \geq
    0, \theta \in \C \right\},
\end{equation}
(``$\op{cl}$'' denotes closure),
and where $\delta \left(T_{\C}(\theta^*)\right)$
denotes the \textit{statistical dimension} of the cone
$T_{\C}(\theta^*)$.
In general, the statistical dimension of a closed
cone $T \subseteq \R^n$ is defined as
\begin{equation}\label{equation:stat_dim}
\delta(T) \coloneqq \E \|\Pi_{T} (Z) \|^2,
\end{equation}
where the expectation is with respect to $Z \sim N(0, I_n)$.
Many properties of the statistical dimension are covered by \citet{amelunxen2014living}.

In the case when the constraint set $\C$ is a subspace, the estimator
$\hat{\theta}(Y)$ is linear and, in this case, it is easy to see that
$\delta(T_{\C}(\theta^*))$ is simply the dimension of $\C$, so that
inequality \eqref{ohub} becomes an equality. For general closed convex
sets, it is therefore reasonable to ask how tight inequality
\eqref{ohub} is. It is not hard to construct examples of $\C$ and
$\theta^* \in \C$ where inequality \eqref{ohub} is loose for fixed $\sigma>0$.
However \citet{oymak2013sharp} proved remarkably that the upper bound in
\eqref{ohub} is tight in the limit as $\sigma \downarrow 0$ (we shall
refer to this in the sequel as the \textit{low $\sigma$ limit});
that is, when $Z \sim N(0,I_n)$,
\begin{equation}\label{equation:oymakhassibi}
  \lim_{\sigma \downarrow 0} \frac{1}{\sigma^2} R(\hat{\theta},
  \theta^*) = \delta (T_{\C}(\theta^*)).
\end{equation}
In summary, \citet{oymak2013sharp} proved that $\sigma^2
\delta(T_{\C}(\theta^*))$ is a nice formula for the risk of
$\hat{\theta}(Y)$ that is, in general, an upper bound which is
tight in the low $\sigma$ limit.

We remark that although \citet{oymak2013sharp} state the results
\eqref{ohub} and \eqref{equation:oymakhassibi}
for the specific case $Z \sim N(0, I_n)$,
their proof automatically extends to the more general setting where $Z$ is
an arbitrary zero mean random vector with $\E \norm{Z}^2 < \infty$
(the components $Z_1, \dots, Z_n$ of $Z$ can be arbitrarily
dependent), provided we generalize the definition \eqref{equation:stat_dim}
of statistical dimension by taking the expectation
with respect to $Z$, without assuming $Z$ is standard Gaussian.
We refer to this modification of the definition \eqref{equation:stat_dim}
as the
\textit{generalized statistical dimension} of the cone $T$.
As a slight abuse of notation, we use the same notation $\delta(\cdot)$ for
this more general concept, with the understanding that the expectation in the definition
is with respect to the distribution of $Z$.
By dropping the Gaussian assumption,
the generalized statistical dimension loses much of the
interpretability and nice geometric properties of the usual statistical dimension
\cite{amelunxen2014living}, but still serves as an abstract notion of the size of a cone $T$
with respect to a distribution $Z$.

This paper deals with the behavior of the estimator $\hat{\theta}(Y)$
when the assumption $\theta^* \in \C$ is violated. We shall refer to
the situation when $\theta^* \notin \C$ as the \textit{misspecified}
setting. Note that, in practice, one can never know if the unknown
$\theta^*$ truly lies in $\C$. It is therefore necessary to study the
behavior of $\hat{\theta}(Y)$ under misspecification.

For the misspecified setting, one must first note that it is no longer
reasonable to measure the performance of $\hat{\theta}(Y)$ by the risk
$R(\hat{\theta}, \theta^*)$, simply because $\hat{\theta}(Y)$
is constrained to be in $\C$ and hence
cannot be expected to be close to $\theta^*$ which is essentially
unconstrained. There are two natural notions of accuracy of
$\hat{\theta}(Y)$ in the misspecified setting, which we call
the \textit{misspecified risk} and the \textit{excess risk}.
The misspecified risk is defined as
\begin{equation}\label{equation:misspecified_risk}
M(\hat{\theta}, \theta^*) \coloneqq  \E_{\theta^*} \|\hat{\theta}(Y) -
\Pi_{\C}(\theta^*) \|^2,
\end{equation}
and the excess risk is defined as
\begin{equation}\label{equation:excess_risk}
E(\hat{\theta}, \theta^*) \coloneqq  \E_{\theta^*} \|\hat{\theta}(Y) -
\theta^*\|^2 - \|\Pi_{\C}(\theta^*) - \theta^* \|^2.
\end{equation}
The misspecified risk, $M(\hat{\theta}, \theta^*)$, is motivated by
the observation that, in the misspecifed case, the estimator
$\hat{\theta}(Y)$ is really estimating $\Pi_{\C}(\theta^*)$ so it is
natural to measure its squared distance from $\Pi_{\C}(\theta^*)$. On
the other hand, the excess risk, $E(\hat{\theta}, \theta^*)$, measures
the squared distance of the estimator from $\theta^*$ relative to the squared
distance of $\Pi_{\C}(\theta^*)$ from $\theta^*$. We refer the reader to
\citet{bellec2015sharpshape} and \autoref{section:basic_properties}
for some background and basic properties
on these notions of accuracy under misspecification. For example, it
can be shown that $M(\hat{\theta}, \theta^*)$ is always less than or
equal to $E(\hat{\theta}, \theta^*)$ (see \eqref{equation:pythagorean_inequality}).
It is easy to see that both of
these risk measures equal $R(\hat{\theta}, \theta^*)$ in the
well-specifed case i.e.,
\begin{equation}
  R(\hat{\theta}, \theta^*) =
  M(\hat{\theta}, \theta^*) = E(\hat{\theta}, \theta^*),
  \qt{when $\theta^* \in \C$}.
\end{equation}

An analogue to inequality \eqref{ohub} for the case of
misspecification has been proved by
\citet[Corollary 2.2]{bellec2015sharpshape}, who showed that
\begin{equation}\label{equation:bellec_ub}
  \frac{1}{\sigma^2} M(\hat{\theta}, \theta^*)
  \leq \frac{1}{\sigma^2} E(\hat{\theta}, \theta^*)
  \leq \delta(T_{\C}(\Pi_{\C}(\theta^*))).
\end{equation}
Again, although this was originally stated for $Z \sim N(0, I_n)$,
it holds for arbitrary zero mean random vectors $Z$ with $\E
\norm{Z}^2 < \infty$. Note the similarity between the right-hand sides
of the inequalities \eqref{ohub} and \eqref{equation:bellec_ub}. The
only difference is that the tangent cone at $\theta^*$  is replaced by
the tangent cone at $\Pi_{\C}(\theta^*)$ in the case of
misspecification. Moreover, in the well-specified setting, the above
inequality \eqref{equation:bellec_ub} reduces to \eqref{ohub}.

It is now very natural to ask if the second inequality in \eqref{equation:bellec_ub} is
tight in the low $\sigma$ limit. One might guess that this should be
the case given the result \eqref{equation:oymakhassibi} for the
well-specified setting. However, it turns out that \eqref{equation:bellec_ub} is not
sharp in the low $\sigma$ limit. The main contribution of this paper
is to provide an exact formula for the low $\sigma$ limit of
$M(\hat{\theta}, \theta^*)$ and $E(\hat{\theta},
\theta^*)$ when $\C$ is polyhedral. Specifically, in
\autoref{theorem:lb_polyhedron}, we prove that if the noise $Z$ is
zero mean with $\E \norm{Z}^2 < \infty$ and if $\C$ is polyhedral,
then
\begin{equation}\label{ikey}
\lim_{\sigma \downarrow 0} \frac{1}{\sigma^2} M(\hat{\theta},
\theta^*) = \lim_{\sigma \downarrow 0} \frac{1}{\sigma^2}
E(\hat{\theta}, \theta^*) = \delta\parens*{
    T_{\C}(\Pi_{\C}(\theta^*)) \cap
    (\theta^* - \Pi_{\C}(\theta^*))^{\perp}
},
\end{equation}
where $v^{\perp} \coloneqq \{u \in \R^n : \left<u, v \right> = 0\}$ for
vectors $v \in \R^n$. As we remarked earlier, in most
applications, the constraint set $\C$ is polyhedral.

Because the set $T_{\C}(\Pi_{\C}(\theta^*)) \cap
(\theta^* - \Pi_{\C}(\theta^*))^{\perp}$ is a subset of
$T_{\C}(\Pi_{\C}(\theta^*))$, the right hand side of \eqref{ikey} is
never larger than $\delta(T_{\C}(\Pi_{\C}(\theta^*)))$.
Under the assumption that the polyhedron $\C$ has a nonempty interior
along with a mild condition on the noise $Z$,
it can be proved that
the right hand side of \eqref{ikey} is \textit{strictly} smaller than
$\delta(T_{\C}(\Pi_{\C}(\theta^*)))$ when $\theta^* \notin \C$ (an
even stronger statement is proved in \autoref{lemma:jump}),
which then implies that
$\lim_{\sigma \downarrow 0} \sigma^{-2} M(\hat{\theta}, \theta^*) <
\lim_{\sigma \downarrow 0} \sigma^{-2} R(\hat{\theta},
\Pi_{\C}(\theta^*))$. This inequality is more interpretable in the
following form:
\begin{equation}\label{equation:risk_gap}
  \lim_{\sigma \downarrow 0} \frac{1}{\sigma^2} \E_{\theta^*}
  \|\hat{\theta}(Y) - \Pi_{\C}(\theta^*) \|^2 < \lim_{\sigma
    \downarrow 0} \frac{1}{\sigma^2} \E_{\Pi_{\C}(\theta^*)}
  \|\hat{\theta}(Y) - \Pi_{\C}(\theta^*) \|^2 \qt{whenever $\theta^*
    \notin \C$}.
\end{equation}
Inequality \eqref{equation:risk_gap} can be qualitatively understood as
follows. The left hand side above corresponds to misspecification
where the data are generated from $\theta^* \notin \C$ while the right
hand side corresponds to the well-specified setting where the data are
generated from $\Pi_{\C}(\theta^*)$. Note that in both cases, the
estimator $\hat{\theta}(Y)$ is really estimating $\Pi_{\C}(\theta^*)$
so it is natural to compare the squared expected distance to
$\Pi_{\C}(\theta^*)$ in both situations. The interesting aspect is
that (in the low $\sigma$ limit) the expected squared distance is
smaller in the
misspecified setting compared to the well-specified setting. To the
best of our knowledge, this fact has not been noted in the literature
previously at this level of generality.

Our main result, \autoref{theorem:lb_polyhedron}, is stated and proved in
\autoref{section:lb_polyhedron}
where some intuition is also
provided for the exact form of the low $\sigma$ limit in
misspecification. The low $\sigma$ limit can be explicitly computed in
certain specific situations. In \autoref{section:examples}, we
specialize to the Gaussian model $Z \sim N(0, I_n)$
and study
in detail the examples when $\C$ is the nonnegative orthant and when
$\C$ is the monotone cone (this latter case corresponds to isotonic
regression).

In \autoref{section:discussion}, we explore issues naturally related
to \autoref{theorem:lb_polyhedron}. In
\autoref{subsection:generalizing}, we consider the situation when $\C$
is not polyhedral. It seems hard to characterize the low $\sigma$
misspecification limits in this case but it is possible to compute them
when $\C$ is the unit ball. It is interesting to note that the low
$\sigma$ limits of $M(\hat{\theta}, \theta^*)$ and $E(\hat{\theta},
\theta^*)$ are different in this case (in sharp contrast to the
polyhedral situation). In \autoref{subsection:high_noise_limit}, we
deal with the risks when $\sigma$ is large. Under some conditions, it
is possible to write a formula for the large $\sigma$ limits of
$M(\hat{\theta}, \theta^*)$ and $E(\hat{\theta}, \theta^*)$; see
\autoref{proposition:high_sigma_limit}. In
\autoref{subsection:max_risk}, we deal with the maximum normalized
risks:
\begin{equation}\label{equation:max_risks}
  \sup_{\sigma > 0} \frac{1}{\sigma^2} M(\hat{\theta}, \theta^*) \quad
  \text{ and } \quad \sup_{\sigma > 0} \frac{1}{\sigma^2} E(\hat{\theta},
  \theta^*).
\end{equation}
In the well-specified setting, inequalities \eqref{ohub} and
\eqref{equation:oymakhassibi} together imply that the maximum
normalized risk equals $\delta(T_{\C}(\theta^*))$. However in the
misspecified setting, the quantities \eqref{equation:max_risks} lie between  $\delta(T_{\C}(\Pi_{\C}(\theta^*)) \cap
(\theta^* - \Pi_{\C}(\theta^*))^{\perp})$ and
$\delta(T_{\C}(\Pi_{\C}(\theta^*)))$. It seems hard to write down an
exact formula for the quantities \eqref{equation:max_risks} but we present some
simulation evidence in \autoref{subsection:max_risk} to argue that they
can be strictly between  $\delta(T_{\C}(\Pi_{\C}(\theta^*)) \cap
(\theta^* - \Pi_{\C}(\theta^*))^{\perp})$ and
$\delta(T_{\C}(\Pi_{\C}(\theta^*)))$.

We conclude with an appendix that contains technical lemmas and proofs
of the various intermediate results throughout the paper.

\section{Background and Notation} \label{section:basic_properties}
In this short section, we shall set up some notation and also
recollect standard results in convex analysis that will be used in the
remainder of the paper.

For $x \in \R^n$ and $r > 0$, we denote by $B_r(x) \coloneqq \{u \in \R^n:
\|u - x\| \leq r\}$ the closed ball of radius $r$ centered at $x$. For
$v \in \R^n$, let $v^\perp \coloneqq \{u \in \R^n : \inner{u,v}=0\}$
denote the hyperplane with normal vector $v$. For $\theta_0 \in \C$,
let $F_\C(\theta_0) \coloneqq \{\theta - \theta_0 : \theta \in \C\}$ be
the result of re-centering the set $\C$ about $\theta_0$. Also recall
the definition of the tangent cone
\eqref{tcd} and note that $T_\C(\theta_0) = \op{cl}\{\alpha x : x \in
F_\C(\theta_0), \alpha > 0\}$.

If $A$ is an $m \times n$ matrix and $J \subseteq \{1,\ldots,m\}$, we let $a_j$ denote the $j$th row of $A$, and let $A_J$ denote the matrix obtained by combining the rows of $A$ indexed by $J$.

A \textit{polyhedron} refers to a set of the form $\{x \in \R^n : Ax
\le b\}$ for some $A \in \R^{m \times n}$ and $b \in \R^n$ where the
inequality $\leq$ is interpreted coordinate-wise, i.e. $\inner{a_j, x} \le b_j$
for $j=1,\ldots,m$.
We will assume that no two pairs $(a_j, b_j)$ and $(a_k, b_k)$ are scalar multiples of each other.
A \textit{polyhedral
  cone} is a set of the form $\{x \in \R^n : A x \leq 0\}$ for some $A
\in \R^{m \times n}$.
Again, we will assume that no two rows of $A$ are scalar multiples of each other.
A \textit{face} of a polyhedron refers to any
subset obtained by setting some of the polyhedron's linear inequality
constraints to equality instead.

In the remainder of this section, we shall collect some standard
results above convex projections that will be used in the paper. These
results can be found in a standard reference such as
\cite{hiriart2012fundamentals}.  Recall that $\Pi_{\C}(x)$ denotes the
projection of a vector $x \in \R^n$ on a closed convex set $\C$. It is
well known that $\Pi_\C(x)$ is the unique vector in $\C$ satisfying
the optimality condition
\begin{equation}\label{equation:optimality condition}
\inner{z - \Pi_\C(x), x - \Pi_\C(x)} \le 0,
\qquad \forall z \in \C.
\end{equation}
Consequently, we have the following Pythagorean inequality
\begin{equation}
\norm{z - x}^2
= \norm{z - \Pi_\C(x)}^2 + \norm{\Pi_\C(x) - x}^2 + 2\inner{z - \Pi_\C(x), \Pi_\C(x) - x}
\ge \norm{z - \Pi_\C(x)}^2 + \norm{\Pi_\C(x) - x}^2.
\end{equation}
Plugging in $z=\Pi_\C(y)$ and $x = \theta^*$ shows that the
misspecified error is
upper bounded by the excess error, that is,
\begin{equation}\label{equation:pythagorean_inequality}
\norm{\Pi_\C(y) - \Pi_\C(\theta^*)}^2
\le \norm{\Pi_\C(y) - \theta^*}^2 - \norm{\Pi_\C(\theta^*) - \theta^*}^2,
\qquad \forall y \in \R^n.
\end{equation}
If instead we plug in $z = \Pi_\C(\theta^*)$ to \eqref{equation:optimality condition}, we have $\inner{\Pi_\C(x) - \Pi_\C(\theta^*), x - \Pi_\C(x)} \ge 0$, which implies
\begin{align}
\norm{\Pi_\C(x) - \theta^*}^2 - \norm{\Pi_\C(\theta^*) - \theta^*}^2
&= - \norm{\Pi_\C(x) - \Pi_\C(\theta^*)}^2
+ 2 \inner{\Pi_\C(x) - \Pi_\C(\theta^*), \Pi_\C(x) - \theta^*}
\\
&\le - \norm{\Pi_\C(x) - \Pi_\C(\theta^*)}^2
+ 2 \inner{\Pi_\C(x) - \Pi_\C(\theta^*), x - \theta^*}
\\
&\le \norm{x - \theta^*}^2.
\end{align}
Combining this with \eqref{equation:pythagorean_inequality}, we see that for $Y = \theta^* + \sigma Z$ we have
\begin{equation}\label{equation:dom_conv}
0
\le \norm{\Pi_\C(Y) - \Pi_\C(\theta^*)}^2
\le \norm{\Pi_\C(Y) - \theta^*}^2 - \norm{\Pi_\C(\theta^*) - \theta^*}^2
\le \sigma^2 \norm{Z}^2.
\end{equation}

In the special case where $\C$ is a cone, the optimality condition
\eqref{equation:optimality condition} implies that $\Pi_\C(x)$ is the
unique vector in $\C$ satisfying
\begin{equation}\label{equation:optimality_condition_cone}
\inner{\Pi_\C(x), x - \Pi_\C(x)} = 0,
\quad \text{ and } \quad
\inner{z, x - \Pi_\C(x)} \le 0,
\quad \forall z \in \C.
\end{equation}

\section{Main theorem: low noise limit for polyhedra}\label{section:lb_polyhedron}

Our main result below provides a precise characterization of the low
$\sigma$ limits of the risks \eqref{equation:misspecified_risk} and
\eqref{equation:excess_risk} (normalized by $\sigma^2$) in the
misspecified setting (i.e.,  when $\theta^* \notin \C$) for polyhedral
$\C$. An implication of this result is that the low $\sigma$ limit can
be much smaller than the upper bound \eqref{equation:bellec_ub} of
\citet{bellec2015sharpshape}.

\begin{theorem}[Low noise limit of risk for polyhedra]
\label{theorem:lb_polyhedron}
Let $\C \subseteq \R^n$ be a closed convex set,
and let
$Y = \theta^* + \sigma Z$
where $\theta^* \in \R^n$ is not necessarily in $\C$,
and $Z$ is zero mean with $\E \norm{Z}^2 < \infty$.
Suppose the following ``locally polyhedral'' condition holds.
\begin{align}
\begin{split}\label{split:polyhedral_condition}
&T_\C(\Pi_\C(\theta^*))
\text{ is a polyhedral cone, and }\\
&T_\C(\Pi_\C(\theta^*)) \cap B_{r^*}(0)
= F_\C(\Pi_\C(\theta^*)) \cap B_{r^*}(0)
\text{ for some $r^*>0$.}
\end{split}
\end{align}
Then,
\begin{align}
\lim_{\sigma \downarrow 0} \frac{1}{\sigma^2} M(\hat{\theta},\theta^*)
= \lim_{\sigma \downarrow 0} \frac{1}{\sigma^2} E(\hat{\theta},\theta^*)
= \delta\parens*{
    T_\C(\Pi_\C(\theta^*)) \cap
    (\theta^*-\Pi_\C(\theta^*))^\perp
}.
\label{equation:lb_polyhedron_stat_dim}
\end{align}
\end{theorem}

Note again that $\delta(\cdot)$ denotes the generalized statistical dimension
induced by the noise $Z$,
and reduces to the usual statistical dimension \cite{amelunxen2014living}
when $Z \sim N(0, I_n)$.

We remark that the ``locally polyhedral'' condition
\eqref{split:polyhedral_condition} essentially states that $\C$ looks
like a polyhedron in a neighborhood around $\Pi_\C(\theta^*)$. As
established in the following lemma, it automatically holds if $\C$ is a
polyhedron, so one can
replace any mention of condition \eqref{split:polyhedral_condition}
with ``$\C$ is a polyhedron'' for the sake of readability. We provide
some remarks on the case when $\C$ is not polyhedral in
\autoref{subsection:generalizing}.

\begin{lemma}\label{lemma:drop_constraints}
Let $\C$ be a polyhedron.
Then the locally polyhedral condition \eqref{split:polyhedral_condition}
holds for any $\theta^* \in \R^n$.
\end{lemma}

Next, the following lemma establishes that the set
$T_\C(\Pi_\C(\theta^*)) \cap (\theta^*-\Pi_\C(\theta^*))^\perp$
that appears in the limit \eqref{equation:lb_polyhedron_stat_dim}
is a face of the tangent cone $T_\C(\Pi_\C(\theta^*))$.

\begin{lemma}\label{lemma:hyperplane_intersection_face}
Let $\theta^* \in \R^n$ and let $\C \subseteq \R^n$ be a closed convex
set satisfying the locally polyhedral condition
\eqref{split:polyhedral_condition}.
Let $A \in \R^{m \times n}$ be such that
$T_\C(\Pi_\C(\theta^*)) = \{u : Au \le 0\}$.
Then there exists some subset $J \subseteq \{1,\ldots,m\}$
such that
\begin{equation}
T_\C(\Pi_\C(\theta^*))
\cap (\theta^* - \Pi_\C(\theta^*))^\perp
= \{u : A_Ju=0, A_{J^c} u \le 0\}.
\end{equation}
Thus, $T_\C(\Pi_\C(\theta^*))
\cap (\theta^* - \Pi_\C(\theta^*))^\perp$
is a face of $T_\C(\Pi_\C(\theta^*))$.
\end{lemma}

Both the above lemmas are proved in
\autoref{section:lemmas_lb_polyhedron}.

If $\theta^* \in \C$ then we have $\Pi_\C(\theta^*)=\theta^*$,
and \autoref{theorem:lb_polyhedron} reduces to the result
\eqref{equation:oymakhassibi} of \citet{oymak2013sharp}: the excess
risk and the misspecified risk become the same, and the common limit
is the statistical dimension of $T_\C(\theta^*)$. We must remark here
that the result of \citet{oymak2013sharp} holds for non-polyhedral
$\C$ as well. We discuss the non-polyhedral setting further in
\autoref{subsection:generalizing}.

\autoref{theorem:lb_polyhedron} states that in the misspecified
case $\theta^* \notin \C$, the low sigma limit still involves the
tangent cone $T_\C(\Pi_\C(\theta^*))$, but one needs to intersect it with the
hyperplane $(\theta^*-\Pi_\C(\theta^*))^\perp$ before taking the statistical
dimension. Due to the optimality condition \eqref{equation:optimality condition}
characterizing $\Pi_\C$,
the tangent cone lies entirely
on one side of the hyperplane, so the hyperplane does not intersect the
interior of the tangent cone.
Therefore, the interior of the tangent cone $T_{\C}(\Pi_{\C}
(\theta^*))$ does not contribute
to the low $\sigma$ limit of the risk under misspecification. This makes
sense because when $\theta^* \notin \C$ and
$\sigma$ is small, the observation vector $Y$ is outside $\C$
with high probability so that $\hat{\theta}(Y)$ lies on the
boundary of $\C$.

In general, the intersection $T_\C(\Pi_\C(\theta^*)) \cap (\theta^* -
\Pi_\C(\theta^*))^\perp$  can be anything from $\{0\}$ to the full
tangent cone $T_\C(\Pi_\C(\theta^*))$ and so the low sigma limit can
be anything between $0$ and $\delta(T_\C(\Pi_\C(\theta^*)))$.  The
case when the limit equals zero corresponds to the situation where
$\theta^*$ lies in the interior of the preimage of $\Pi_\C(\theta^*)$
under the map $\Pi_\C$ so that every point in some neighborhood of
$\theta^*$ is projected onto the same point $\Pi_\C(\theta^*)$ (see
\autoref{subfigure:orthant_zero} for an example).

The following lemma
(proved in \autoref{section:lemmas_lb_polyhedron}),
provides mild conditions under which the intersection
$T_\C(\Pi_\C(\theta^*)) \cap (\theta^* - \Pi_\C(\theta^*))^\perp$
has strictly smaller generalized statistical dimension
than the full tangent cone $T_\C(\Pi_\C(\theta^*))$.

\begin{lemma}\label{lemma:jump}
Let $\C \subseteq \R^n$ be a polyhedron with nonempty interior.
Then
\begin{equation}
\sup_{\theta^* \notin \C : \Pi_\C(\theta^*)=\theta_0}
\delta\parens*{T_\C(\Pi_\C(\theta^*)) \cap (\theta^* - \Pi_\C(\theta^*))^\perp}
< \delta(T_\C(\theta_0)).
\end{equation}
for every $\theta_0 \in \C$,
provided the random vector $Z$
has nonzero probability of lying in the interior of
$T_\C(\theta_0)$.
\end{lemma}

As mentioned already, \autoref{lemma:jump}
combined with the main result \autoref{theorem:lb_polyhedron}
implies the risk gap \eqref{equation:risk_gap}.
In summary, under the nonempty interior assumption,
if we think of the low $\sigma$ limit as a function of $\theta^*$, we
see that as $\theta^*$ approaches $\C$ from the outside there is a
``jump'' when $\theta^*$ enters $\C$. This ``jump'' phenomenon is not
unique to the polyhedral case. In \autoref{subsection:generalizing}
we discuss a non-polyhedral example that also exhibits this jump
phenomenon.

\autoref{theorem:lb_polyhedron} suggests something that may seem
nonintuitive: if $\theta^* \notin \C$ and we use the estimator
$\hat{\theta}(Y) = \Pi_\C(Y)$, the risk when $Y = \theta^* + \sigma Z$
is smaller than the risk when $Y = \Pi_\C(\theta^*) + \sigma Z$.
As mentioned already,
in the case $Y = \theta^* + \sigma Z$ the estimator is actually estimating $\Pi_\C(\theta^*)$,
not $\theta^*$.
Moreover,
the risks \eqref{equation:misspecified_risk} and \eqref{equation:excess_risk}
measure error relative to $\Pi_\C(\theta^*)$ rather than to $\theta^*$.
Furthermore, the intuition is that in the low $\sigma$ limit, the
estimator $\hat{\theta}(Y)$ in the misspecified setting is a
projection onto a much smaller set than in the well-specified setting
(essentially, a face of a tangent cone instead of the full tangent
cone), so more of the original noise in $Y$ is eliminated.
This qualitatively explains why having $Y$ generated from $\theta^*$
outside $\C$ allows the estimator to estimate $\Pi_\C(\theta^*)$
better than if $Y$ were generated from $\Pi_\C(\theta^*)$ instead.

Finally, we observe that in the misspecified setting, there is a gap
between Bellec's upper bound $\delta(T_{\C}(\Pi_{\C}(\theta^*)))$
\eqref{equation:bellec_ub}
and the low
$\sigma$ risk limit, unlike in the well-specified setting where
the result \eqref{equation:oymakhassibi} implies that the normalized risk
increases to the upper bound in the low $\sigma$ limit. The upper
bound, which is constant in $\sigma$, can become very loose as $\sigma
\downarrow 0$. However, in \autoref{subsection:max_risk} we shown a
few examples where the normalized risk is close to the upper bound for
some $\sigma$, as well as examples where the normalized risk remains
much smaller than the upper bound for all $\sigma>0$.

\subsection{Proof of \autoref{theorem:lb_polyhedron}}
\label{subsection:proof_lb_polyhedron}

We establish one key lemma (proved in \autoref{section:lemmas_lb_polyhedron})
before proving \autoref{theorem:lb_polyhedron}.
It is a deterministic result that contains the core of the argument:
roughly, if we have a polyhedral cone $\T$
and any $\theta^* \in \R^n$ satisfying
$\Pi_\T(\theta^*) = 0$,
then any point $u$ sufficiently near $\theta^*$ will have its projection $\Pi_\T(u)$
lying in the hyperplane with normal direction $\theta^*$.

\begin{lemma}[Key lemma]\label{lemma:project_hyperplane}
Fix $\theta^* \in \R^n$,
and let $\T$ be a closed convex set such that the re-centered set
$\{\theta - \Pi_\T(\theta^*) : \theta \in \T\}$
is a polyhedral cone.
Then there exists $r>0$ such that
\begin{equation}\label{equation:project_hyperplane}
    \Pi_\T(u) - \Pi_\T(\theta^*) \in (\theta^* - \Pi_\T(\theta^*))^\perp,
    \qquad \forall u \in B_r(\theta^*).
\end{equation}
\end{lemma}

With this lemma,
along with some standard results collected in \autoref{section:basic_properties},
we can proceed with proving \autoref{theorem:lb_polyhedron}.
\begin{proof}[Proof of \autoref{theorem:lb_polyhedron}]
We first prove
\begin{equation}\label{equation:proof_part1}
\lim_{\sigma \downarrow 0} \frac{1}{\sigma^2} M(\hat{\theta}, \theta^*)
= \delta\parens*{
    T_\C(\Pi_\C(\theta^*)) \cap (\theta^* - \Pi_\C(\theta^*))^\perp
}.
\end{equation}
For any $r>0$ we can write
\begin{equation}\label{equation:tail}
\frac{1}{\sigma^2} M(\hat{\theta}, \theta^*)
= \frac{1}{\sigma^2} \E_{\theta^*} \brackets*{
    \norm{\Pi_\C(Y) - \Pi_\C(\theta^*)}^2
    \b{1}_{\{Y \in B_r(\theta^*)\}}
}
+
\frac{1}{\sigma^2} \E_{\theta^*}\brackets*{
    \norm{\Pi_\C(Y) - \Pi_\C(\theta^*)}^2
    \b{1}_{\{Y \notin B_r(\theta^*)\}}
}.
\end{equation}
We claim the second term on the right-hand side vanishes as $\sigma \downarrow 0$ (regardless of the value of $r>0$).
Since the projection $\Pi_\C$ is non-expansive \cite{hiriart2012fundamentals},
\begin{equation}
0 \le
\frac{1}{\sigma^2} \E_{\theta^*}\brackets*{
    \norm{\Pi_\C(Y) - \Pi_\C(\theta^*)}^2
    \b{1}_{\{Y \notin B_r(\theta^*)\}}
}
\le \frac{1}{\sigma^2} \E_{\theta^*}\brackets*{
    \norm{Y - \theta^*}^2
    \b{1}_{\{Y \notin B_r(\theta^*)\}}
}
= \E_{\theta^*} \brackets*{\norm{Z}^2 \b{1}_{\{\sigma\norm{Z} > r\}}}.
\end{equation}
Then, the dominated convergence theorem implies the right-hand side tends to zero as $\sigma \downarrow 0$, because $\E \norm{Z}^2 < \infty$ and the random variable $\norm{Z}^2 \b{1}_{\{\sigma\norm{Z} > r\}}$ converges to zero pointwise.

Thus, it remains to show
\begin{equation}\label{equation:main_term}
\lim_{\sigma \downarrow 0}
\frac{1}{\sigma^2} \E_{\theta^*} \brackets*{
    \norm{\Pi_\C(Y) - \Pi_\C(\theta^*)}^2
    \b{1}_{\{Y \in B_r(\theta^*)\}}
}
=
\delta\parens*{
    T_\C(\Pi_\C(\theta^*)) \cap (\theta^* - \Pi_\C(\theta^*))^\perp
}
\end{equation}
for some $r>0$.

We define the re-centered tangent cone
\begin{equation}
\T \coloneqq \braces*{\Pi_\C(\theta^*) + u : u \in T_\C(\Pi_\C(\theta^*))}.
\end{equation}
We claim there exists some $r>0$ such that
\begin{equation}\label{equation:project_onto_cone}
\Pi_\C(u) = \Pi_\T(u),
\qquad \forall u \in B_r(\theta^*).
\end{equation}
Indeed, note that the locally polyhedral condition \eqref{split:polyhedral_condition}
implies the existence of some $r^*>0$ such that
\begin{equation}\label{equation:polyhedron_condition_shift}
\C \cap B_{r^*}(\Pi_\C(\theta^*)) = \T \cap B_{r^*}(\Pi_\C(\theta^*))
\end{equation}
Since both projections $\Pi_\C$ and $\Pi_\T$ are continuous \cite{hiriart2012fundamentals} at $\theta^*$,
there exists some $r>0$ such that the image of $B_r(\theta^*)$ under both projections
lies in $B_{r^*}(\Pi_\C(\theta^*))$.
Thus the local equality \eqref{equation:project_onto_cone} of the projections
follows from the locally polyhedral condition
\eqref{equation:polyhedron_condition_shift}.

By combining this argument with \autoref{lemma:project_hyperplane},
we have shown there exists some $r > 0$
that satisfies not only \eqref{equation:project_onto_cone},
but also \eqref{equation:project_hyperplane}.
With this value of $r$, the equality \eqref{equation:project_onto_cone} implies that replacing each instance of $\C$ with $\T$ in \eqref{equation:main_term}
does not change either side, since $\Pi_\C(Y)=\Pi_\T(Y)$, $\Pi_\C(\theta^*)=\Pi_\T(\theta^*)$,
and
\begin{equation}
T_\C(\Pi_\C(\theta^*))
= T_{\C \cap B_{r^*}(\Pi_\C(\theta^*))}(\Pi_\C(\theta^*))
= T_{\T \cap B_{r^*}(\Pi_\C(\theta^*))}(\Pi_\C(\theta^*))
= T_{\T}(\Pi_\T(\theta^*)),
\end{equation}
by the equality \eqref{equation:polyhedron_condition_shift}
and the definition of the tangent cone.
Thus it remains to prove
\begin{equation}\label{equation:main_term2}
\lim_{\sigma \downarrow 0}
\frac{1}{\sigma^2} \E_{\theta^*} \brackets*{
    \norm{\Pi_\T(Y) - \Pi_\T(\theta^*)}^2
    \b{1}_{\{Y \in B_r(\theta^*)\}}
}
=
\delta(\K),
\end{equation}
where $\K \coloneqq T_\T(\theta^*) \cap (\theta^* - \Pi_\T(\theta^*))^\perp$.

Since $r$ satisfies \eqref{equation:project_hyperplane}, some re-centering yields
\begin{equation}\label{equation:projection_trick}
\Pi_\T(Y) - \Pi_\T(\theta^*)
= \Pi_{T_\T(\theta^*)}(Y-\Pi_\T(\theta^*))
= \Pi_\K(Y - \Pi_\T(\theta^*))
\end{equation}
in the event $\{Y \in B_r(\theta^*)\}$.

For
$W \coloneqq (\theta^* - \Pi_\T(\theta^*))^\perp$, we claim
\begin{equation}
\Pi_\K = \Pi_\K \circ \Pi_W.
\end{equation}
In fact this holds for any subspace $W$ and closed convex $\K \subseteq W$,
by the Pythagorean theorem:
\begin{equation}
\Pi_\K(x)
= \argmin_{u \in \K} \norm{x - u}^2
= \argmin_{u \in \K} \braces*{
    \norm{x - \Pi_W(x)}^2 + \norm{\Pi_W(x) - u}^2
}
= \Pi_\K (\Pi_W(x)).
\end{equation}
Applying this to \eqref{equation:projection_trick} yields
\begin{align}
\Pi_\T(Y) - \Pi_\T(\theta^*)
&= \Pi_\K(Y - \Pi_\T(\theta^*))
\\
&= \Pi_\K (\Pi_W(\theta^* + \sigma Z - \Pi_\T(\theta^*)))
\\
&= \Pi_\K(\Pi_W(\sigma Z))
& \text{$\Pi_W$ is linear, $\Pi_W(\theta^* - \Pi_\T(\theta^*)) = 0$}
\\
&= \Pi_\K(\sigma Z) = \sigma \Pi_\K(Z)
& \text{$\K$ is a cone}
\end{align}
in the event $\{Y \in B_r(\theta^*)\}$.
By plugging this into the left-hand side of equation \eqref{equation:main_term2},
we have
\begin{equation}
\lim_{\sigma \downarrow 0}
\E_{\theta^*} \brackets*{
    \norm{\Pi_\K(Z)}^2
    \b{1}_{\{Y \in B_r(\theta^*)\}}
}
= \E \norm{\Pi_\K(Z)}^2
= \delta(\K),
\end{equation}
where the first equality follows by dominated convergence
($\norm{\Pi_\K(Z)}^2 \le \norm{Z}^2$ and $\E \norm{Z}^2 < \infty$).
This verifies the desired equality \eqref{equation:main_term2} and
concludes the proof of the first low $\sigma$ limit \eqref{equation:proof_part1}.

We now prove the other equality
\begin{equation}
\lim_{\sigma \downarrow 0}
\frac{1}{\sigma^2} M(\hat{\theta}, \theta^*)
=
\lim_{\sigma \downarrow 0}
\frac{1}{\sigma^2} E(\hat{\theta}, \theta^*).
\end{equation}
We claim
\begin{equation}\label{equation:excess_tail}
\lim_{\sigma \downarrow 0}
\frac{1}{\sigma^2} \E_{\theta^*}
\brackets*{
    \parens*{\norm{\Pi_\C(Y) - \theta^*}^2 - \norm{\Pi_\C(\theta^*) - \theta^*}^2}
    \b{1}_{\{Y \notin B_r(\theta^*)\}}
}
= 0
\end{equation}
for any $r>0$.
Applying some basic properties \eqref{equation:dom_conv} of the projection $\Pi_\C$ yields
\begin{equation}
0 \le \frac{1}{\sigma^2} \E_{\theta^*} \brackets*{
    \parens*{\norm{\Pi_\C(Y) - \theta^*}^2 - \norm{\Pi_\C(\theta^*) - \theta^*}^2}
    \b{1}_{\{Y \notin B_r(\theta^*)\}}
}
\le \E \brackets*{\norm{Z}^2 \b{1}_{\{\sigma \norm{Z} \ge r\}}},
\end{equation}
so applying the dominated convergence theorem as before leads to the limit \eqref{equation:excess_tail}.

Thus, it suffices to prove
\begin{align}
&\lim_{\sigma \downarrow 0}
\frac{1}{\sigma^2} \E_{\theta^*}
\brackets*{
    \norm{\Pi_\C(Y) - \Pi_\C(\theta^*)}^2
    \b{1}_{\{Y \in B_r(\theta^*)\}}
}
\\
&=
\lim_{\sigma \downarrow 0}
\frac{1}{\sigma^2} \E_{\theta^*}
\brackets*{
    \parens*{\norm{\Pi_\C(Y) - \theta^*}^2 - \norm{\Pi_\C(\theta^*) - \theta^*}^2}
    \b{1}_{\{Y \in B_r(\theta^*)\}}
}
\label{align:main_term_excess}
\end{align}
for some $r > 0$.
We choose $r$ as before so that \eqref{equation:project_hyperplane} and
\eqref{equation:project_onto_cone} both hold.
By the same reasoning as before, we can replace each instance of $\C$ with $\T$ without changing anything.
Furthermore, the condition \eqref{equation:project_hyperplane} implies we have $\inner{\Pi_\T(Y) - \Pi_\T(\theta^*), \theta^* - \Pi_\T(\theta^*)} = 0$ in the event $\{Y \in B_r(\theta^*)\}$,
so the Pythagorean inequality
\eqref{equation:pythagorean_inequality}
becomes equality:
\begin{equation}
\norm{\Pi_\T(Y) - \Pi_\T(\theta^*)}^2 \b{1}_{\{Y \in B_r(\theta^*)\}}
= \parens*{\norm{\Pi_\T(Y) - \theta^*}^2 - \norm{\Pi_\T(\theta^*) - \theta^*}^2}
\b{1}_{\{Y \in B_r(\theta^*)\}}.
\end{equation}
Therefore the equality \eqref{align:main_term_excess} holds, which concludes the proof of \autoref{theorem:lb_polyhedron}.
\end{proof}

\section{Examples}\label{section:examples}

In this section, we assume the Gaussian noise model $Z \sim N(0,
I_n)$, or equivalently $Y \sim N(\theta^*, \sigma^2 I_n)$.
Thus, $\delta(\cdot)$ denotes the usual statistical dimension
\cite{amelunxen2014living}, where $Z$ in the definition \eqref{equation:stat_dim}
is a standard Gaussian vector.

\subsection{Nonnegative orthant}

We now apply \autoref{theorem:lb_polyhedron}
to the \textit{nonnegative orthant}
$\R^n_+ \coloneqq \{u \in \R^n : u_i \ge 0, \forall i\}$.
In \autoref{figure:nonnegative_orthant}
we provide visualizations of the geometry of the main theorem
when applied to this constraint set.

\begin{corollary}[Nonnegative orthant]\label{corollary:orthant}
Let $Y \sim N(\theta^*, \sigma^2 I)$ where $\theta^* \in \R^n$.
Let $n_+ \coloneqq \sum_{i=1}^n \b{1}_{\{\theta^*_i > 0\}}$ and
$n_0 \coloneqq \sum_{i=1}^n \b{1}_{\{\theta^*_i=0\}}$
denote the number of positive components
and number of zero components of $\theta^*$ respectively.
Then the normalized excess risk \eqref{equation:excess_risk}
and normalized mispecified risk \eqref{equation:misspecified_risk}
of the least squares estimator $\hat{\theta}(Y)\coloneqq \Pi_{\R^n_+}(Y)$
with respect to $\R_+^n$
both tend to
\begin{equation}
    \frac{n_0}{2} + n_+
\end{equation}
as $\sigma \downarrow 0$.
\end{corollary}

\begin{proof}
By \autoref{theorem:lb_polyhedron},
it suffices to prove that the statistical dimension term in
\eqref{equation:lb_polyhedron_stat_dim} is $\frac{n_0}{2} + n_+$. Note that for $y \in \R^n$,
$\Pi_{\R^n_+}(y) = \max\{y,0\}$ is obtained
by taking the component-wise maximum of $y$ with $0$.
Consequently,
\begin{equation}
T_{\R^n_+}(\Pi_{\R^n_+}(\theta^*))
= \{u \in \R^n: u_i \ge 0 \text{ if } (\Pi_{\R^n_+}(\theta^*))_i=0\}
= \{u \in \R^n: u_i \ge 0 \text{ if } \theta^*_i \le 0\}.
\end{equation}
Also,
\begin{equation}
(\theta^*-\Pi_{\R^n_+}(\theta^*))^\perp
= \braces*{
    u \in \R^n:
    \sum_{i : \theta^*_i < 0} \theta^*_i u_i=0
}
\end{equation}
The intersection is thus
\begin{align}
T_{\R^n_+}(\Pi_{\R^n_+}(\theta^*))
\cap
(\theta^*-\Pi_{\R^n_+}(\theta^*))^\perp
= \braces*{
    u \in \R^n :
    \begin{aligned}
    u_i \ge 0 && \text{if } \theta^*_i=0\\
    u_i = 0 && \text{if } \theta^*_i < 0
    \end{aligned}
}
\cong
\R^{n_+} \times \R_+^{n_0} \times \{0\}^{n-n_+-n_0}.
\end{align}
The result follows by noting $\delta(\R)=1$ and $\delta(\R_+)=1/2$
and by using the fact that
$\delta(T_1 \times T_2) = \delta(T_1) + \delta(T_2)$
for any two cones $T_1$ and $T_2$ \cite{amelunxen2014living}.
\end{proof}

\begin{remark}
For $\theta^* \in \R^n$ let $n_+$ and $n_0$ be as defined in
\autoref{corollary:orthant}. Then the low $\sigma$ limit for the
corresponding
well-specified problem $Y \sim N(\Pi_{\R^n_+}(\theta^*), \sigma^2 I)$
is $\frac{n-n_+}{2} + n_+$ since all negative components of $\theta^*$
are sent to zero by $\Pi_{\R^n_+}$. This is larger than the low
$\sigma$ limit for the misspecified problem $Y \sim N(\theta^*,
\sigma^2 I)$ because $n-n_+ \ge n_0$, with strict inequality if
$\theta^* \notin \R^n_+$.
\end{remark}

\begin{figure}[!ht]
\centering
\begin{subfigure}[t]{0.3\textwidth}
\centering
\begin{tikzpicture}[scale=1]
\fill[gray!30] (0,0) rectangle (2,2);
\draw[->] (0,0) -- (0,2);
\draw[->] (0,0) -- (2,0);
\fill (1,-1) circle[radius=0.5mm];
\node[right] at (1,-1) {$\theta^*$};
\fill (1,0) circle[radius=0.5mm];
\node[below right] at (1,0) {$\Pi(\theta^*)$};
\draw[<->, line width = 0.4mm] (-1,0) -- (2,0);
\draw[dashed] (1,-1) -- (1,0);
\end{tikzpicture}
\caption{$\theta = (1,-1)$; $\delta=1$}
\end{subfigure}
~
\begin{subfigure}[t]{0.3\textwidth}
\centering
\begin{tikzpicture}[scale=1]
\fill[gray!30] (0,0) rectangle (2,2);
\draw[->] (0,0) -- (0,2);
\draw[->] (0,0) -- (2,0);
\fill (0,-1) circle[radius=0.5mm];
\node[right] at (0,-1) {$\theta^*$};
\fill (0,0) circle[radius=0.5mm];
\node[below right] at (0,0) {$\Pi(\theta^*)$};
\draw[->, line width = 0.4mm] (0,0) -- (2,0);
\draw[dashed] (0,-1) -- (0,0);
\end{tikzpicture}
\caption{$\theta = (0,-1)$; $\delta=1/2$}
\end{subfigure}
~
\begin{subfigure}[t]{0.3\textwidth}
\centering
\begin{tikzpicture}[scale=1]
\fill[gray!30] (0,0) rectangle (2,2);
\draw[->] (0,0) -- (0,2);
\draw[->] (0,0) -- (2,0);
\fill (-1,-1) circle[radius=0.5mm];
\node[right] at (-1,-1) {$\theta^*$};
\fill (0,0) circle[radius=0.8mm];
\node[below right] at (0,0) {$\Pi(\theta^*)$};
\draw[dashed] (-1,-1) -- (0,0);
\end{tikzpicture}
\caption{$\theta = (-1,-1)$; $\delta=0$}
\label{subfigure:orthant_zero}
\end{subfigure}
\caption{$\R^2_+$ is marked by the gray area.
The intersection $T_{\R^2_+}(\Pi_{\R^2_+}(\theta^*)) \cap (\theta^* - \Pi_{\R^2_+}(\theta^*))^\perp$ [translated to be centered at $\Pi_{\R^2_+}(\theta^*)$] is marked by the bold lines in the first two examples, and the bold point in the third example. Each sub-caption states the statistical dimension $\delta = \delta(T_{\R^2_+}(\Pi_{\R^2_+}(\theta^*)) \cap (\theta^* - \Pi_{\R^2_+}(\theta^*))^\perp)$.}
\label{figure:nonnegative_orthant}
\end{figure}
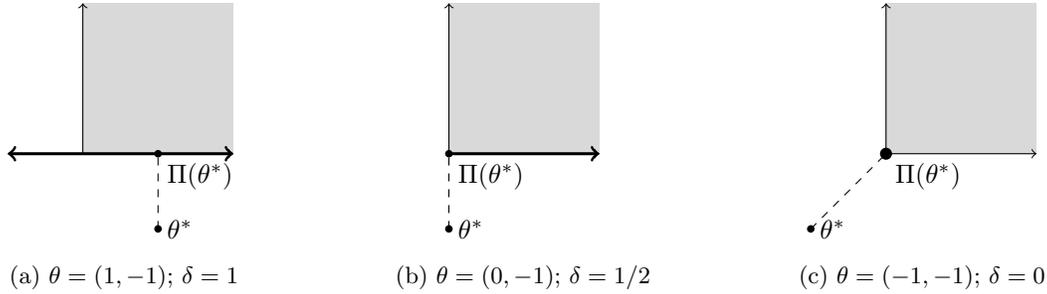

\subsection{Consequences for isotonic regression}\label{subsection:isotonic}
This section details interesting consequences of
\autoref{theorem:lb_polyhedron} for isotonic regression under
misspecification. Let
\begin{equation}
    \S^n \coloneqq \{u \in \R^n : u_1 \le \cdots \le u_n\}
\end{equation}
be the \textit{monotone cone}.
We call elements of $\mathcal{S}^n$ \textit{nondecreasing}.

By a \textit{block}, we refer to a set of the form $\{k, k+1, \dots,
l\}$ for two nonnegative integers $k \leq l$. Consider a partition
of $\{1,\ldots,n\}$ into blocks $I_1,\ldots,I_m$ listed in increasing
order (i.e., the maximum entry of $I_i$ is strictly smaller than the
minimum entry of $I_j$ for $i < j$). Let $\abs{I_j}$ denote the
cardinality of $I_j$ and note that $\sum_{j=1}^m |I_j|= n$ as $I_1,
\dots, I_m$ form a partition of $\{1, \dots, n\}$. Let $\S_{\abs{I_1},
\dots, \abs{_m}}$ denote the induced \textit{block monotone cone}
defined as
\begin{equation}\label{equation:block_monotone}
\S_{\abs{I_1},\ldots,\abs{I_m}}
\coloneqq \{u \in \S^{n} : \text{$u$ is constant on each of the blocks $I_1,\ldots,I_m$}\}
\end{equation}
For example,
\begin{equation}
\S_{2,3,2} = \{u \in \R^{2+3+2} : u_1 = u_2 \le u_3 = u_4 = u_5 \le u_6 = u_7\}.
\end{equation}

\autoref{theorem:lb_polyhedron} implies the following result, which we
prove in \autoref{subsection:proof_isotonic}.

\begin{proposition}[Isotonic regression]\label{proposition:isotonic}
Let $Y \sim N(\theta^*, \sigma^2 I)$ where $\theta^* \in \R^n$.
Let $(J_1,\ldots,J_K)$ be the partition of $\{1,\ldots,n\}$ into blocks such that
$\Pi_{\S^n}(\theta^*)$ is constant on each $J_k$ with
respective values $\mu_1 < \cdots < \mu_K$. For each $k \in
\{1,\ldots, K\}$, there exists a unique finest partition
$(I^k_1,\ldots, I^k_{m_k})$
of $J_k$ into blocks such that for all $j \in
\{1,\ldots,m_k\}$, the mean of the components of
$\theta^*$ on each $I_j^k$ equals $\mu_k$; that is,
\begin{equation}\label{equation:partition_mean}
\frac{1}{\abs*{I^k_j}} \sum_{i \in I^k_j} \theta^*_i = \mu_k,
\qquad 1 \le j \le m_k.
\end{equation}
Then the common low $\sigma$ limit of the
normalized excess risk \eqref{equation:excess_risk} and normalized
misspecified risk \eqref{equation:misspecified_risk} of the isotonic least squares
estimator $\hat{\theta}(Y)\coloneqq \Pi_{\S^n}(Y)$ equals
\begin{equation}\label{equation:isotone_limit}
\sum_{k=1}^K
\delta\parens*{\S_{\abs{I^k_1},\ldots,\abs{I^k_{m_k}}}}.
\end{equation}
\end{proposition}

It is clear from the above proposition that the low $\sigma$ behavior
of the isotonic estimator under misspecification crucially depends
on the statistical dimension of the block monotone cone
$\S_{\abs{I^k_1},\ldots,\abs{I^k_{m_k}}}$.
[We remark again that throughout this section we only deal with
the usual statistical dimension,
where the noise $Z$ in the definition \eqref{equation:stat_dim}
is standard Gaussian.]
Here, we provide two simple properties of the block monotone cone \eqref{equation:block_monotone},
each of which implies that when the block sizes are equal, the statistical dimension is simply that of $\S^{m_k}$.
The first result provides a direct connection to weighted isotonic regression.

\begin{lemma}[Weighted isotonic regression]\label{lemma:weighted_isotonic_regression}
Let $z \in \R^n$ and let $I_1,\ldots,I_m$ be a partition of $\{1,\ldots,n\}$
into blocks.
Let $\bar{z}_{I_j}\coloneqq \frac{1}{\abs{I_j}} \sum_{i \in I_j} z_i$.
Then
$\Pi_{\S_{\abs{I_1},\ldots,\abs{I_m}}}(y)$
is the vector that is constant on the blocks $I_1,\ldots,I_m$
with constant values $x^*_1,\ldots,x^*_m$,
where $x^*=(x^*_1,\ldots,x^*_m)$ is
\begin{equation}
x^* = \argmin_{x \in \S^m} \sum_{j=1}^m \abs{I_j}(x_i-\bar{z}_{I_j})^2.
\end{equation}
In other words, the values on the constant blocks of
$\Pi_{\S_{\abs{I_1},\ldots,\abs{I_m}}}(z)$
can be found by weighted isotonic regression of $(\bar{z}_{I_1},\ldots,\bar{z}_{I_m}) \in \R^m$
with weights $\abs{I_1},\ldots,\abs{I_m}$.

Consequently, when $\abs{I_1}=\cdots=\abs{I_m}$,
the statistical dimension of the block monotone cone is
\begin{equation}
\delta(\S_{\abs{I_1},\ldots,\abs{I_m}})
= \sum_{j=1}^m \frac{1}{j}.
\end{equation}
\end{lemma}

The next lemma shows $\S_{\abs{I_1},\ldots,\abs{I_m}}$ is isometric to a particular cone in the lower-dimensional space $\R^m$.

\begin{lemma}[Block monotone cone isometry]\label{lemma:block_monotone_stat_dim}
The block monotone cone
$\S_{\abs{I_1},\ldots,\abs{I_m}} \subseteq \R^n$
is isometric to
\begin{equation}\label{equation:block_monotone_embed}
\braces*{v \in \R^m : \frac{v_1}{\sqrt{\abs{I_1}}} \le \cdots \le \frac{v_m}{\sqrt{\abs{I_m}}}}\subseteq \R^m,
\end{equation}
and thus both sets have the same statistical dimension.
In particular, if
$\abs{I_1}=\cdots=\abs{I_m}$, then the statistical dimension of the block monotone cone is
\begin{equation}
\delta(\S_{\abs{I_1},\ldots,\abs{I_m}})
= \sum_{j=1}^m \frac{1}{j}.
\end{equation}
\end{lemma}

Both lemmas are proved in \autoref{subsection:proofs_block_monotone}.
Note that for the case $\abs{I_1} = \cdots = \abs{I_m} = 1$,
both lemmas reduce to the statement of the
statistical dimension of the monotone cone $\S^n$ \cite[Eq. D.12]{amelunxen2014living}.
More generally, when the $m$ blocks have equal size,
the statistical dimension of the associated block monotone cone
is the same as that of the monotone cone $\S^m$.
In \autoref{subsection:stat_dim_block_monotone},
we discuss what \autoref{lemma:block_monotone_stat_dim} suggests
for the completely general case when the block sizes are arbitrary.

By combining either of these two lemmas with \autoref{proposition:isotonic},
we immediately obtain an explicit expression for the low $\sigma$ limits in a special case.
For $m \geq 1$, we denote
the harmonic number $\sum_{j=1}^m (1/j)$ by $H_m$.

\begin{corollary}[Isotonic regression with equal sub-block sizes]
Consider the setting of \autoref{proposition:isotonic}.
In the special case where
\begin{equation}\label{equation:special_case}
\abs{I^k_1} = \cdots = \abs{I^k_{m_k}} \qt{for each $k \in \{1,\ldots, K\}$,}
\end{equation}
the common low $\sigma$ limit has the following explicit expression:
\begin{equation}\label{equation:isotone_limit_special}
\sum_{k=1}^K H_{m_k} = \sum_{k=1}^K \sum_{j=1}^{m_k} \frac{1}{j}.
\end{equation}
\end{corollary}

See the examples to follow (as well as \autoref{subsection:stat_dim_block_monotone}) for further discussion about how the statistical dimension of
$\S_{\abs{I^k_1},\ldots,\abs{I^k_{m_k}}}$ behaves in general,
when the special condition \eqref{equation:special_case} does not hold.

In \autoref{table:isotonic}, we demonstrate how to apply this theorem to various cases of $\theta^*$.
In the ``partition of $\theta^*$'' column,
we use square brackets to partition the components of $\theta^*$ into
$K$ blocks according to the constant pieces $\mu_1 < \cdots < \mu_K$
of $\Pi_{\S^n}(\theta^*)$, and then within the $k$th group use
parentheses to further partition the components into $m_k$ sub-blocks
each with common mean $\mu_k$.

\def\arraystretch{1.25}
\begin{table}[!ht]
\small
\centering
\begin{tabular}{ccccc}
$\theta^*$ & $\Pi_{\S^n}(\theta^*)$ & partition of $\theta^*$ & $m_1,\ldots,m_K$ & $\sum_{k=1}^K H_{m_k}$
\\ \hline\hline
$(0,0,0,0,0,0)$ & $(0,0,0,0,0,0)$ & $[(0),(0),(0),(0),(0),(0)]$ & $6$ & $H_6=2.45$
\\
$(1,-1,1,-1,1,-1)$ & $(0,0,0,0,0,0)$ & $[(1,-1),(1,-1),(1,-1)]$ & $3$ & $H_3=1.8\bar{3}$
\\
$(5,3,1,-1,-3,-5)$ & $(0,0,0,0,0,0)$ & $[(5,3,1,-1,-3,-5)]$ & $1$ & $H_1=1$
\\ \hline
$(-1,-1,-1, -1, 2, 2)$ & $(-1,-1,-1,-1,2,2)$ & $[(-1), (-1), (-1), (-1)], [(2), (2)]$ & $4,2$ & $H_4+H_2 = 3.58\bar{3}$
\\
$(0,-2,1,-3,2,2)$ & $(-1,-1,-1,-1,2,2)$ & $[(0,-2), (1,-3)], [(2), (2)]$ & $2,2$ & $H_2+H_2 = 3$
\\
$(0,0,-2,-2,3,1)$ & $(-1,-1,-1,-1,2,2)$ & $[(0,0,-2,-2)], [(3,1)]$ & $1,1$ & $H_1+H_1=2$
\end{tabular}
\caption{Examples of how to compute the limit in \autoref{proposition:isotonic} in the special case \eqref{equation:special_case}.}
\label{table:isotonic}
\end{table}

We now discuss in detail what \autoref{proposition:isotonic} states for certain cases of $\theta^*$.

\begin{enumerate}
    \item
    In the well-specified case where $\theta^* \in \S^n$,
    we have $\theta^*_j = \mu_k$ for all $j \in J_k$ and $k \in \{1,\ldots,K\}$,
    so the finest partition of each $J_k$ is the partition into singleton sets.
    Then $m_k = \abs{J_k}$ for each $k$,
    and moreover $\abs{I^k_j}=1$ for all valid $k$ and $j$.
    Thus, \autoref{proposition:isotonic} implies that
    both low $\sigma$ limits are
    \begin{equation}\label{equation:isotone_limit_ws}
        \sum_{k=1}^K H_{\abs{J_k}}
        \coloneqq \sum_{k=1}^K \sum_{j=1}^{\abs{J_k}} \frac{1}{j},
    \end{equation}
    This is precisely the upper bound \eqref{equation:bellec_ub} for the monotone cone
    as computed by \citet[Prop. 3.1]{bellec2015sharpshape},
    so we recover the low $\sigma$ limit \eqref{equation:oymakhassibi}.
    Computations for the well-specified examples $\theta^* =(0,0,0,0,0,0)$ and $\theta^*=(-1,-1,-1,-1,2,2)$
    appear in \autoref{table:isotonic}.

    Now, consider the misspecified problem $Y \sim N(\theta^*, \sigma^2 I_n)$
    with $\theta^* \notin \mathcal{S}^n$,
    and compare the statement of \autoref{proposition:isotonic}
    with the corresponding statemetn for the
    well-specified problem $Y \sim N(\Pi_{\S^n}(\theta^*),\sigma^2 I)$.
    In both cases, the partition of $\{1,\ldots,n\}$ into $(J_1,\ldots,J_K)$ is the same.
    However, we showed above that in the well-specified problem,
    the sub-partition of each $J_k$ consists of singletons,
    whereas for the misspecified problem we may get nontrivial partitions $(I^k_1,\ldots,I^k_{m_k})$.
    Noting the inclusion
    $\S_{\abs{I^k_1},\ldots,\abs{I^k_{m_k}}} \subseteq \S^{\abs{J_k}}$
    for each $k$ and comparing \eqref{equation:isotone_limit}
    for the two cases yields
    \begin{equation}
    \sum_{k=1}^K \delta(\S_{\abs{I^k_1},\ldots,\abs{I^k_{m_k}}})
    \le \sum_{k=1}^K \delta(\S^{\abs{J_k}}),
    \end{equation}
    which shows that in general the misspecified low $\sigma$ limit is smaller than
    the corresponding well-specified limit.

    \item
    Suppose $\theta^*$ is nonincreasing and nonconstant i.e., $\theta^*
    \in (-\S^n) \setminus \S^n$. Then $\Pi_{\S^n}(\theta^*)$ is
    constant (see \cite{RWD88} for various properties of
    $\Pi_{\S^n}$), so $K=1$ and $\mu_1 = \frac{1}{n} \sum_{i=1}^n
    \theta^*_i$.  We also claim $m_1=1$. Indeed, if $m_1 > 1$ then
    there exists some $j < n$ such that
    $\mu = \frac{1}{j} \sum_{i=1}^j \theta^*_i = \frac{1}{n-j} \sum_{i=j+1}^n \theta^*_i$.
    However, the fact that $\theta^*$ is nonincreasing and nonconstant implies
    $\frac{1}{j} \sum_{i=1}^j \theta^*_i > \frac{1}{n-j} \sum_{i=j+1}^n \theta^*_i$,
    a contradiction.
    Thus, \autoref{proposition:isotonic} implies that both low $\sigma$ limits are $1$.
    (In fact, by combining the above argument with the proof of \autoref{proposition:isotonic},
    we have shown that the intersection
    $T_{\S^n}(\Pi_{\S^n}(\theta^*)) \cap (\theta^* - \Pi_{\S^n}(\theta^*))^\perp$
    is simply the subspace of constant sequences.)
    On the other hand, since $\Pi_{\S^n}(\theta^*)$ is constant, the low $\sigma$ limit
    in the well-specified setting $Y \sim N(\Pi_{\S^n}(\theta^*),
    \sigma^2 I_n)$ is $\sum_{j=1}^n \frac{1}{j} \asymp \log n$, which
    is much larger.

    The logarithmic term appears here in
    the well-specified case due to the well-known spiking effect of
    isotonic regression (documented, for example,
    by \citet{pal2008spiking, wu2015penalized, Zhang02}). Indeed,
    the isotonic estimator is inconsistent near
    the end points which leads to the logarithm term in the risk.
    However, in the misspecified case when $\theta^*$ is
    nonincreasing and nonconstant,
    a combination of
    the proof of \autoref{theorem:lb_polyhedron}
    (in particular \autoref{lemma:project_hyperplane})
    with the fact that
    $T_{\S^n}(\Pi_{\S^n}(\theta^*)) \cap (\theta^* - \Pi_{\S^n}(\theta^*))^\perp$
    is the subspace of all constant sequences
    implies
    $\hat{\theta}(Y)$ is a constant sequence with probability increasing to $1$ as $\sigma \downarrow 0$,
    in which case the constant value must be the sample mean
    $\bar{Y} \coloneqq \frac{1}{n} \sum_{i=1}^n Y_i$.
    Alternatively, one can rephrase the geometric argument
    in \autoref{lemma:project_hyperplane}
    more simply in this example; when $\sigma$ is small, $Y$ is near $\theta^*$
    and thus is also nondecreasing with high probability, in which case $\hat{\theta}(Y)$ is constant, due to the properties of the projection $\Pi_{\mathcal{S}^n}$.
    Hence, in this situation the estimator does not
    suffer from any spiking at the endpoints, and consequently there are
    no logarithmic terms in the risk in the misspecified case in the
    low sigma limit.

    Computations for the specific example when $\theta^* =
    (5,3,1,-1,-3,-5)$ appear in \autoref{table:isotonic}.

    \item In the first half of \autoref{table:isotonic} we consider
      three choices for $\theta^*$ that project to
      $\Pi_{\S^n}(\theta^*)=(0,0,0,0,0,0)$.
    Here $K=1$ and the sub-block sizes $\abs{I^1_1},\ldots,\abs{I^1_{m_1}}$ are equal in each case (namely, the common block size is $1$, $2$, and $6$ respectively), so we are in the special case \eqref{equation:special_case}. Thus, the limit is $\sum_{j=1}^{m_1} \frac{1}{j}$ where $m_1$ is the number of sub-blocks. We see that for the misspecified $\theta^*$ the low $\sigma$ limits are smaller.

    One can heuristically interpret \autoref{theorem:lb_polyhedron}
    for the example $\theta^* = (1,-1,1,-1,1,-1)$ as follows.
    With probability increasing to $1$ as $\sigma \downarrow 0$,
    the estimator $\hat{\theta}(Y)$ is nondecreasing
    and piecewise constant on three equally sized blocks, so the low $\sigma$ limit is
    the same as if we were estimating $(0,0,0)$ in $\S^3$.

    \item Similarly in the second half of \autoref{table:isotonic}
    we consider three $\theta^*$ that project to $\Pi_{\S^n}(\theta^*)=(-1,-1,-1,-1,2,2)$.
    Here $K=2$
    but, since the low $\sigma$ limit decomposes, we can simply consider
    each constant piece separately.     Again, we see that the more
    sub-blocks $I^j_i$, the higher the statistical dimension, with the
    well-specified case having the most sub-blocks (all singletons).

    \item The concrete examples we have considered so far have been in
      the special case \eqref{equation:special_case}. In a few other
      cases we can still provide the low $\sigma$ limit.
      (See also \autoref{subsection:stat_dim_block_monotone} for further discussion.)

    \begin{enumerate}
        \item
        If $K=1$ and $m_1=2$, then the low $\sigma$ limit is $\delta(S_{\abs{I_1^1},\abs{I_2^1}})$.
        By \autoref{lemma:block_monotone_stat_dim}, this is the same as the statistical dimension of the half space $\{u \in \R^2 : u_1/\sqrt{\abs{I_1}} \le u_2/\sqrt{\abs{I_2}}\}$, which is $1.5$.
        However, when $m_1>2$, it is difficult to compute
        $\delta(S_{\abs{I_1^1},\ldots \abs{I_{m_1}^1}})$ unless we are
        in the special case $\abs{I_1^1} = \cdots = \abs{I_{m_1}^1}$.

        \item
        In some other extreme cases we can get an approximation. For example, if
        \begin{equation}\label{equation:sin_like_theta}
        \theta^* = (0,\underbrace{1,\ldots,1}_{(n-2)/2},\underbrace{-1,\ldots,-1}_{(n-2)/2},0),
        \end{equation}
        then $\Pi_{\S^n}(\theta^*)=(0,\ldots,0)$, so the low $\sigma$ limit is $\delta(\S_{1,n-2,1})$.
        \autoref{lemma:block_monotone_stat_dim} shows that this is the same as the statistical dimension of $\{u \in \R^3 : u_1 \le u_2/\sqrt{n-2} \le u_3\}$.
        As $n \to \infty$ tends to this set tends to $\{u \in \R^3 : u_1 \le 0 \le u_3\}$ which has statistical dimension $1+\frac{1}{2}+\frac{1}{2}=2$. Thus $\delta(\S_{1,n-2,1}) \to 2$ as $n \to \infty$.
        We used simulations to verify that
        the low $\sigma$ limit is indeed near $2$ even for $n=20$.

    \end{enumerate}
\end{enumerate}

\section{Further discussion}\label{section:discussion}

\subsection{Generalizing \autoref{theorem:lb_polyhedron} to the non-polyhedral case}
\label{subsection:generalizing}

Note that \autoref{theorem:lb_polyhedron} requires the condition
\eqref{split:polyhedral_condition}  i.e., that $\C$ is locally a
polyhedron near $\Pi_\C(\theta^*)$. Here we comment on the situation
when $\C$ is non-polyhedral. Although non-polyhedral convex sets can
be approximated by polyhedra, the low $\sigma$ limit magnifies the
local geometry of the set and ignores the goodness of such an
approximation. As a stark counterexample, consider
any closed convex $\C \subseteq \R^2$ with nonempty interior,
and let $Z \sim N(0, I_n)$.
For any polygon in $\R^2$, \autoref{theorem:lb_polyhedron}
implies that the low $\sigma$ limits are either $0$, $1/2$, or $1$
because in $\R^2$ the intersection of a convex cone with a line
intersecting the origin is either the origin, a ray, or a line.
Thus, for a sequence of polygons approximating $\C$
the sequence of corresponding low $\sigma$ limits need not even have a
limit, never mind the matter of two different sequences of polygonal
approximations having a common limit. Therefore, the low $\sigma$ limit for
general $\C$ cannot be found using a polyhedral approximation.

In order to understand how the low $\sigma$ limits behave for general
$\C$, we consider the following specific example. Let $\C \coloneqq
\{\theta \in \R^n : \norm{\theta} \le 1\}$ be the unit ball so that
$\Pi_{\C}(x) = \frac{x}{\max\{\norm{x}, 1\}}$. Also let $\theta^* \coloneqq
(r, 0, \dots, 0)$ for some $r > 1$ so that
$\Pi_{\C}(\theta^*) = (1, 0, \dots, 0)$. By rotational symmetry of $\C$,
the case of any general $\theta^* \notin \C$ can be reduced to this
case.

In the corresponding well-specified case $Y \sim N(\Pi_\C(\theta^*),\sigma^2 I_n)$), the result
\eqref{equation:oymakhassibi} of \citet{oymak2013sharp} implies that
the normalized misspecified risk \eqref{equation:misspecified_risk}
and the normalized excess risk \eqref{equation:excess_risk}
are equal in the low $\sigma$ limit with common value
\begin{equation}\label{equation:ball_ws}
\delta(T_{\C}(\Pi_\C(\theta^*)))
=n - \frac{1}{2},
\end{equation}
since the tangent cone is the half space
$T_{\C}(\Pi_\C(\theta^*))=\{x \in \R^n : x_1 \le 0\}$.

However, in the misspecified case, we observe some new phenomena that
do not occur for polyhedra.

\begin{proposition}[Low noise limits for the ball]\label{proposition:ball}
Let $\C\coloneqq \{\theta \in \R^n : \norm{\theta}_2 \le 1\}$, $\theta^* \notin \C$, and $Y \sim N(\theta^*, \sigma^2 I_n)$.
For the estimator $\hat{\theta}(Y)=\Pi_\C(Y)$, we have
\begin{subequations}
\begin{align}
    \lim_{\sigma \downarrow 0}
    \frac{1}{\sigma^2}
    M(\hat{\theta}, \theta^*)
    &=
    \frac{n-1}{\norm{\theta^*}^2},
    \label{align:ball_limit1}
    \\
    \lim_{\sigma \downarrow 0}
    \frac{1}{\sigma^2}
    E(\hat{\theta}, \theta^*)
    &=
    \frac{n-1}{\norm{\theta^*}}.
    \label{align:ball_limit2}
\end{align}
\end{subequations}
\end{proposition}
The proof involves direct computation and appears in \autoref{section:ball_proof}.

We now highlight some of the interesting behavior.
In the polyhedral case, both limits were equal; in the proof of \autoref{theorem:lb_polyhedron}
(in particular \autoref{lemma:project_hyperplane})
we showed that with probability increasing to $1$ (in the low $\sigma$ limit), $Y$ would be projected onto the hyperplane $(\theta^* - \Pi_{\mathcal{C}}(\theta^*))^\perp$,
producing the orthogonality required for the Pythagorean inequality \eqref{equation:pythagorean_inequality} to become an equality.
In the general case, the Pythagorean inequality is not tight,
and we explicitly see from this example that even in the low noise limit the the excess risk can be strictly larger than the misspecified risk.

Note that in contrast to the corresponding well-specified case
$Y \sim N(\Pi_\C(\theta^*),\sigma^2 I_n)$ which has limit $n-\frac{1}{2}$,
the misspecified limits $\frac{n-1}{\norm{\theta^*}^2}$ and $\frac{n-1}{\norm{\theta^*}}$ both tend to $n-1$ as $\norm{\theta^*} \downarrow 1$, so there is a ``jump'' in the limits between the misspecified and well-specified setting.
This is also a feature of \autoref{theorem:lb_polyhedron} when the
polyhedron $\C$ has nonempty interior, as we discussed earlier (see
\autoref{lemma:jump}).

This example shows that \autoref{theorem:lb_polyhedron} does not hold
for nonpolyhedral constraint sets $\C$, as the two normalized risks are
not equal in this particular example of the unit ball, and moreover
neither limit equals
\begin{equation}
\delta(T_\C(\Pi_\C(\theta^*)) \cap (\theta^* -
\Pi_\C(\theta^*))^\perp)
= \delta(\{u : \inner{u,\theta^*} \le 0\}\cap (\theta^*)^\perp)
= \delta((\theta^*)^\perp)
= n-1.
\end{equation}
The intuition for \autoref{theorem:lb_polyhedron} is that, in
the polyhedral case, the projections of $Y$ largely end up in some
face of the polyhedron $\C$, which can be approximated by a
lower-dimensional cone, for which the statistical dimension is well
defined. When $\C$ is not polyhedral, the generalization of this
``face'' is hard to conceptualize
and is likely not well approximated by a cone, so a statistical dimension can not be even applied.
Indeed, for general $\C$ such as the ball, tangent cones are extremely poor approximations for the set. Contrary to this drawback, the result \eqref{equation:oymakhassibi} of \citet{oymak2013sharp} shows that tangent cones are good enough for the well-specified setting. However for the misspecified setting, we expect that any general result for the low $\sigma$ limits does not involve a statistical dimension of some cone, since the surface of $\C$ is the essential object of interest and cannot be approximated by some cone except in special settings like the polyhedral case.

As mentioned already, \autoref{theorem:lb_polyhedron} shows that in the misspecified setting, the upper bound \eqref{equation:bellec_ub}, which holds for all $\sigma$, is not tight in the low $\sigma$ limit. One might ask whether a better upper bound for all $\sigma$ can be achieved, but \autoref{figure:sigma_trace_ball} shows that for some values of $\sigma$ the risks can be close to the upper bound, represented by the solid horizontal line.
We observed this behavior in other examples (see also \autoref{figure:sigma_trace_orthant}): the risks can be close to the upper bound for some moderate values of $\sigma$, and then converge to the strictly smaller low $\sigma$ limit.
Replacing the upper bound \eqref{equation:bellec_ub}, which is constant in $\sigma$,
with a $\sigma$-dependent upper bound would be an interesting result, but it would have to be extremely dependent on the geometry of the set $\C$.
In the following sections we further discuss the normalized risks as a function of $\sigma$.

\begin{figure}[!ht]
\centering
\includegraphics[width=0.3\textwidth]{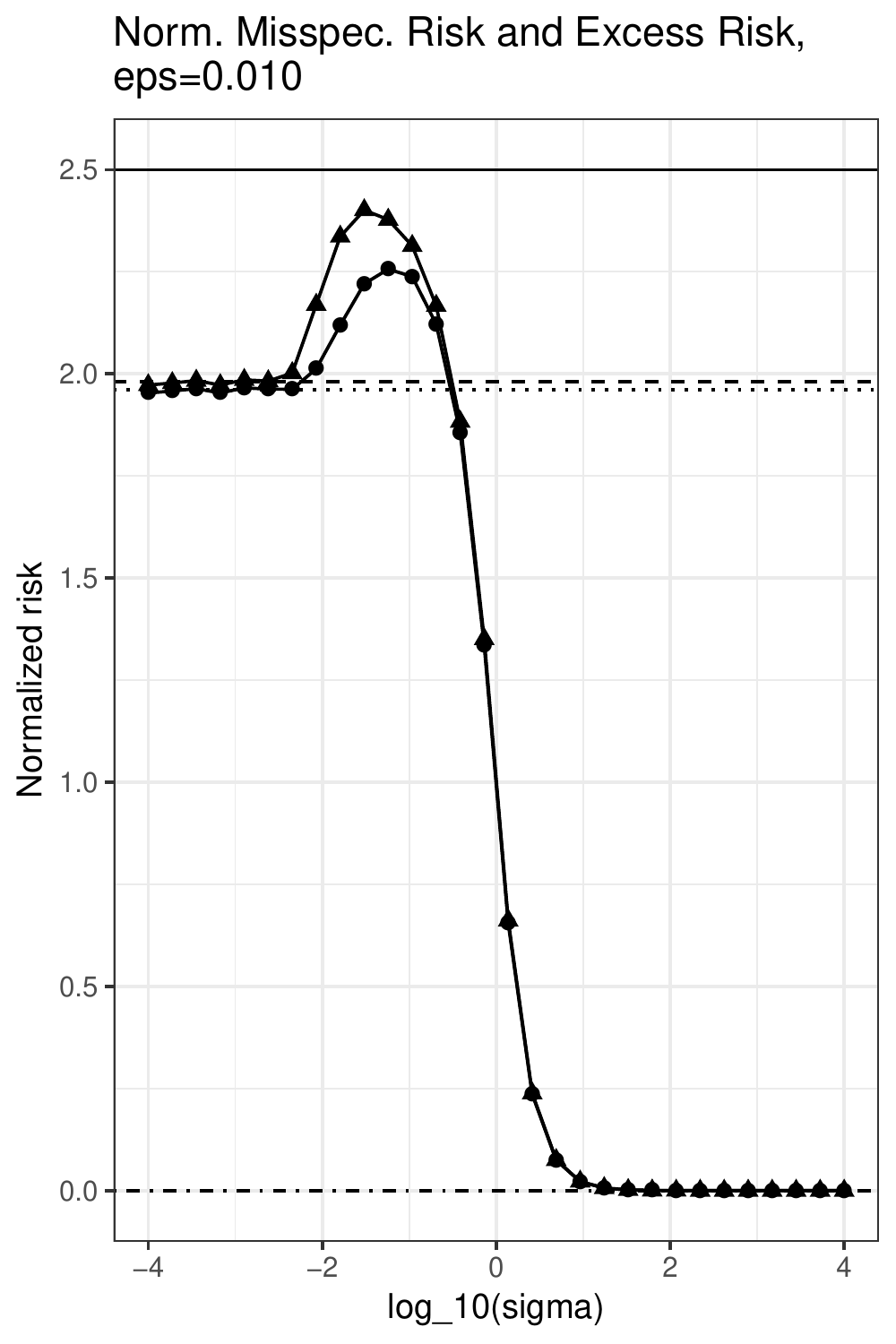}\;
\includegraphics[width=0.3\textwidth]{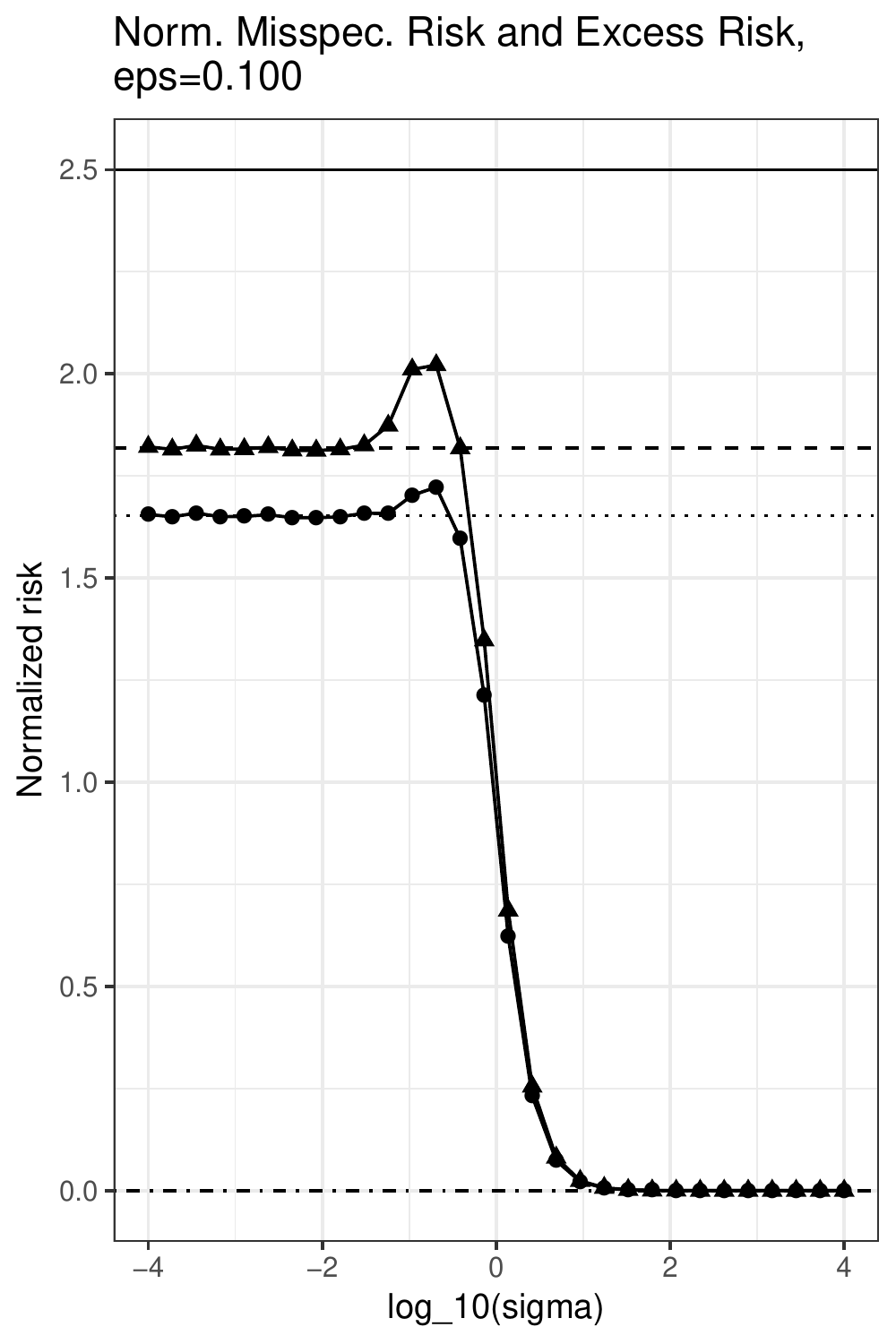}\;
\includegraphics[width=0.3\textwidth]{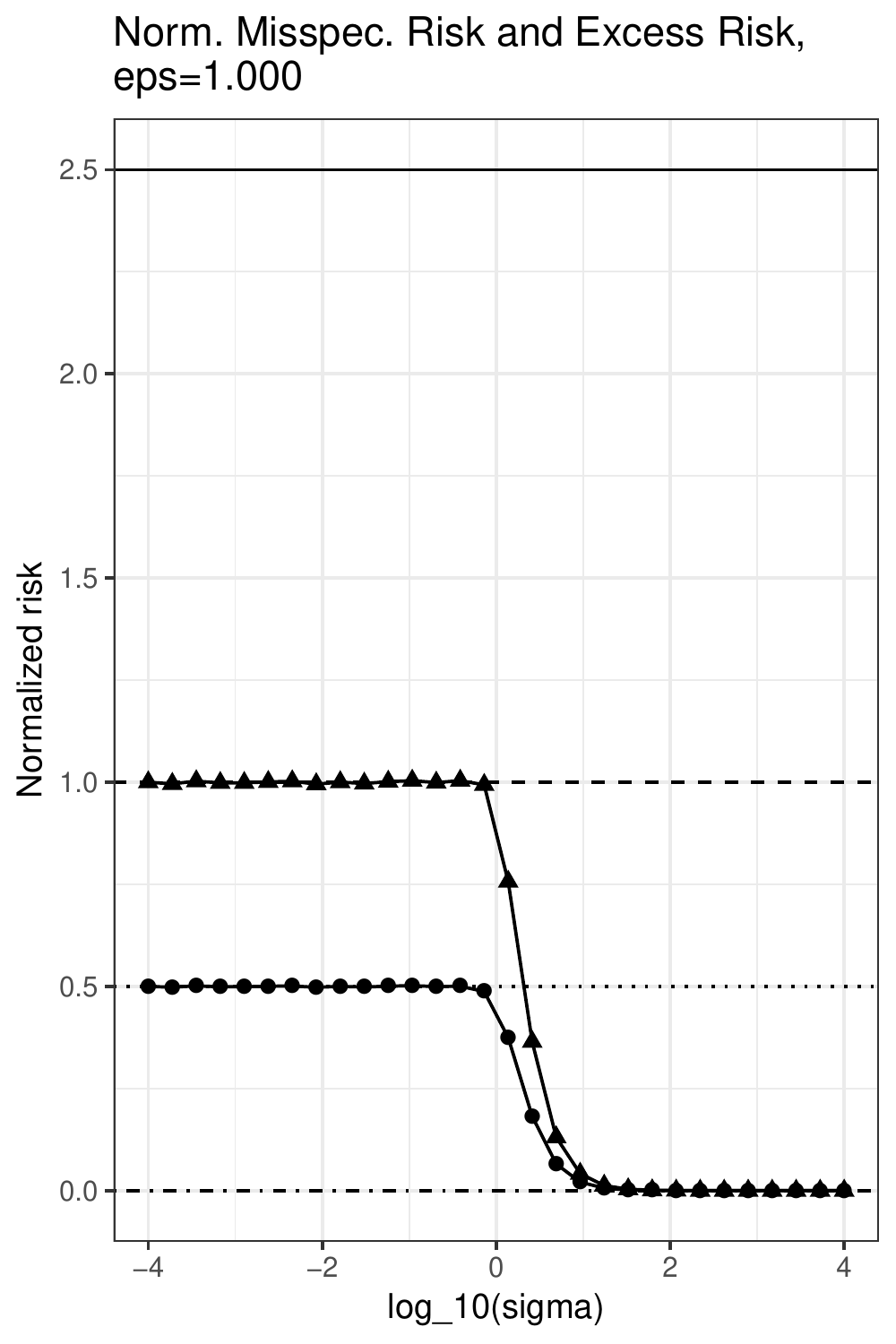}
\caption{
    Empirical estimates of the normalized misspecified risk ($\bullet$)
    and normalized excess risk ($\blacktriangle$) plotted against $\log_{10} (\sigma)$, for the ball $\C=\{\theta \in \R^n : \norm{\theta} \le 1\}$ in the case $n=3$ with $\theta^* = (1+\epsilon,0,0)$ and $\epsilon \in \{0.01, 0.1, 1\}$.
    The solid horizontal line represents the upper bound $\delta(T_\C(\Pi_\C(\theta^*)))=n-\frac{1}{2}=2.5$ guaranteed by \eqref{equation:bellec_ub}.
    The dotted lines and dashed lines are the predicted low $\sigma$ limits $\frac{n-1}{(1+\epsilon)^2}$
    and $\frac{n-1}{1+\epsilon}$ respectively.
    The dash-dot line is the high $\sigma$ limit $0$.
}
\label{figure:sigma_trace_ball}
\end{figure}

\subsection{High noise limit}\label{subsection:high_noise_limit}

Although not interesting in its own right, the high noise limit of the normalized risks
can help characterize the maximum risk as we discuss in the following section.
Proofs for this section appear in \autoref{section:proofs_high_noise_limit}.

For a closed convex set $\C$ we define the \textit{core cone}
\begin{equation}\label{equation:def_core_cone}
K_\C \coloneqq \bigcap_{\theta \in \C} T_\C(\theta).
\end{equation}
Recall the notation for the re-centered set
$F_\C(\theta_0)=\{\theta-\theta_0 : \theta \in \C\}$ where $\theta_0 \in \C$.
For a vector $v \in \R^n$ we let $\R_+ v \coloneqq \{\alpha u : \alpha \ge 0\}$.
We have the following equivalent characterizations of the core cone.

\begin{lemma}[Characterizations of the core cone]\label{lemma:core_cone}
Let $\C \subseteq \R^n$ be a closed convex set.
For any $\theta_0 \in \C$,
\begin{equation}
K_\C
\overset{(i)}{=}
\{v : \R_+ v \subseteq F_\C(\theta_0)\}
\overset{(ii)}{=}
\bigcap_{\sigma > 0} \frac{F_\C(\theta_0)}{\sigma}.
\end{equation}
Additionally, the inclusion $K_\C \subseteq T_\C(\theta)$
holds for any $\theta \in \C$.
If furthermore $F_\C(\theta_0)$ is a cone, then the equality $K_\C = T_\C(\theta)$ holds
if and only if $\theta_0-(\theta - \theta_0) \in \C$; in particular, taking $\theta = \theta_0$ shows that $K_\C = T_\C(\theta_0) = F_\C(\theta_0)$.
\end{lemma}

Thus, up to a translation, the core cone can either be viewed as the result of shrinking $\C$ radially toward $\theta_0 \in \C$, or as the largest cone centered at $\theta_0 \in \C$ that is contained in $\C$. An interesting point is that $\theta_0 \in \C$ can be chosen arbitrarily.

Furthermore, in case when $\C$ is a cone, the core cone $K_\C$ is this cone $\C$, and we can characterize which tangent cones are the ``smallest'' in the sense that they equal the intersection \eqref{equation:def_core_cone} of all tangent cones.

The following result shows that under a boundedness condition, the core cone characterizes both high $\sigma$ limits.

\begin{proposition}[High noise limit]\label{proposition:high_sigma_limit}
Let $\C$ be a closed convex set.
Let $\theta^* \in \R^n$
and $Y \coloneqq \theta^* + \sigma Z$
where $Z$ is a zero mean random vector with $\E \norm{Z}^2 < \infty$.
If the condition
\begin{equation}\label{equation:boundedness}
\sup_{x \in \R^n} \left( \norm{\Pi_{F_\C(\Pi_\C(\theta^*))}(x)}^2
- \norm{\Pi_{K_\C}(x)}^2 \right) < \infty.
\end{equation}
holds,
then
\begin{equation}
\lim_{\sigma \to \infty} \frac{1}{\sigma^2}
M(\hat{\theta}, \theta^*)
= \lim_{\sigma \to \infty} \frac{1}{\sigma^2}
E(\hat{\theta}, \theta^*)
= \delta(K_\C).
\end{equation}
\end{proposition}

The main hurdle in applying \autoref{proposition:high_sigma_limit} is
verifying the condition \eqref{equation:boundedness}. The following
result covers two cases where it is easy to verify the
condition.

\begin{corollary}[Orthant and bounded sets]\label{corollary:high_sigma_limit_ex}
Let $\theta^* \in \R^n$ and $Y \sim N(\theta^*,\sigma^2 I_n)$.

\begin{itemize}
    \item If $\C = \R^n_+$ is the nonnegative orthant, then the high $\sigma$ limits are $\delta(\R^n_+)=n/2$.
    \item Let $\C$ be a closed convex set.
    $K_\C=\{0\}$ if and only if $\C$ is bounded, in which case both high $\sigma$ limits are $0$.
\end{itemize}
\end{corollary}

\autoref{figure:sigma_trace_ball} and \autoref{figure:sigma_trace_orthant}
illustrate the result of this corollary.

Verifying \eqref{equation:boundedness} for more general $\C$ is more difficult.
We believe it might hold for polyhedral cones with any $\theta^*$, in
which case \autoref{proposition:high_sigma_limit} would imply that the
high $\sigma$ limits are $\delta(\C)$.
An interesting feature of the examples presented thus far is that the
high $\sigma$ limits (including the veracity of
\eqref{equation:boundedness}) do not depend on $\theta^*$.

\begin{remark}\label{remark:cone_case}
More generally, suppose $\C$ is a general cone.
By applying \autoref{lemma:core_cone} with $\theta_0 = 0$ and $\theta = \Pi_\C(\theta^*)$,
we observe that the core cone $K_\C$ is $\C$,
and moreover $\C \subseteq T_\C(\Pi_\C(\theta^*))$,
with equality if and only if $-\Pi_\C(\theta^*) \in \C$.
Thus, if the condition \eqref{equation:boundedness} holds,
then \autoref{proposition:high_sigma_limit} implies the high $\sigma$ limits are $\delta(\C)$,
and moreover \autoref{lemma:core_cone} implies that these limits
equal Bellec's upper bound \eqref{equation:bellec_ub},
$\delta(T_\C(\Pi_\C(\theta^*)))$,
if and only if
$\theta^*$ satisfies $-\Pi_\C(\theta^*) \in \C$.
\end{remark}

However, the condition \eqref{equation:boundedness} does not hold for all $\C$.
One can verify numerically that the epigraph $\C \coloneqq \{u \in \R^2 : u_2 \ge u_1^2\}$,
whose core cone is $K_\C = \{(0,u_2) : u_2 \ge 0\}$, does not satisfy \eqref{equation:boundedness}.
Simulations also show that the high $\sigma$ limits are larger than $\delta(K_\C)=1/2$.
In general, it is unclear exactly when
the core cone does or does not characterize the high $\sigma$ limits.

\subsection{Maximum normalized risk}\label{subsection:max_risk}
Our low and high $\sigma$ limit results \autoref{theorem:lb_polyhedron}
and \autoref{proposition:high_sigma_limit}
provides an incomplete characterization of the maximum normalized risks
\eqref{equation:max_risks}.
As mentioned already in \eqref{equation:bellec_ub},
$\delta(T_\C(\Pi_\C(\theta^*)))$ is an upper bound for both suprema.

In the well-specified case $\theta^* \in \C$, both suprema reduce to the usual normalized risk
$\sigma^{-2} R(\hat{\theta},\theta^*)$;
moreover the upper bound becomes $\delta(T_\C(\theta^*))$, and is attained as $\sigma \downarrow 0$
by the result \eqref{equation:oymakhassibi} of \citet{oymak2013sharp}.

However, in the misspecified case we have shown in \autoref{theorem:lb_polyhedron}
that in general
the low $\sigma$ limit does not attain the upper bound \eqref{equation:bellec_ub}.
Moreover, simulations show that in some cases even the suprema do not attain the upper bound;
see \autoref{figure:sigma_trace_ball} and \autoref{figure:sigma_trace_orthant}. We see that for some cases the suprema are close to the upper bound, but for others it is much smaller.

\begin{figure}[!ht]
\centering
\includegraphics[width=0.3\textwidth]{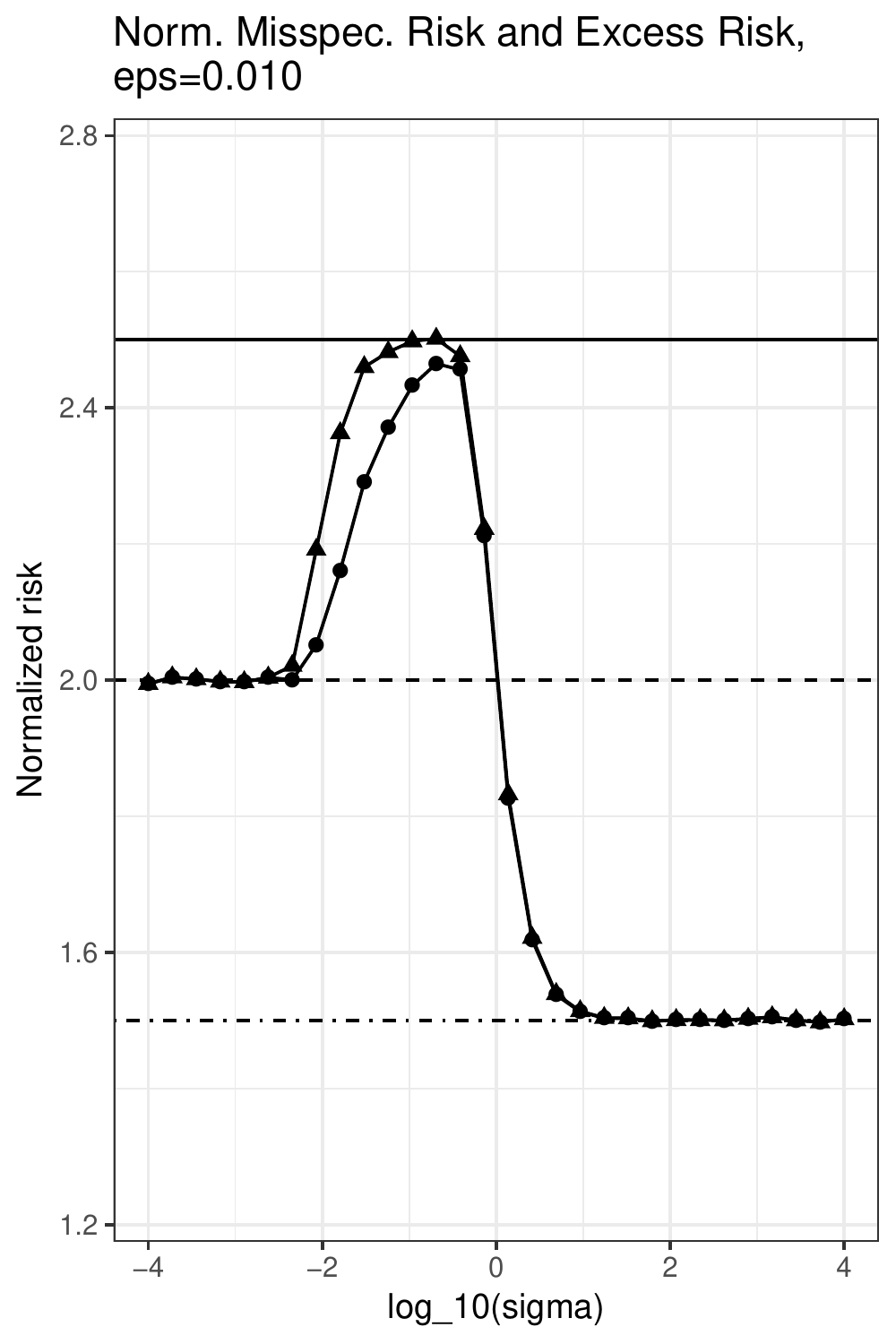}\;
\includegraphics[width=0.3\textwidth]{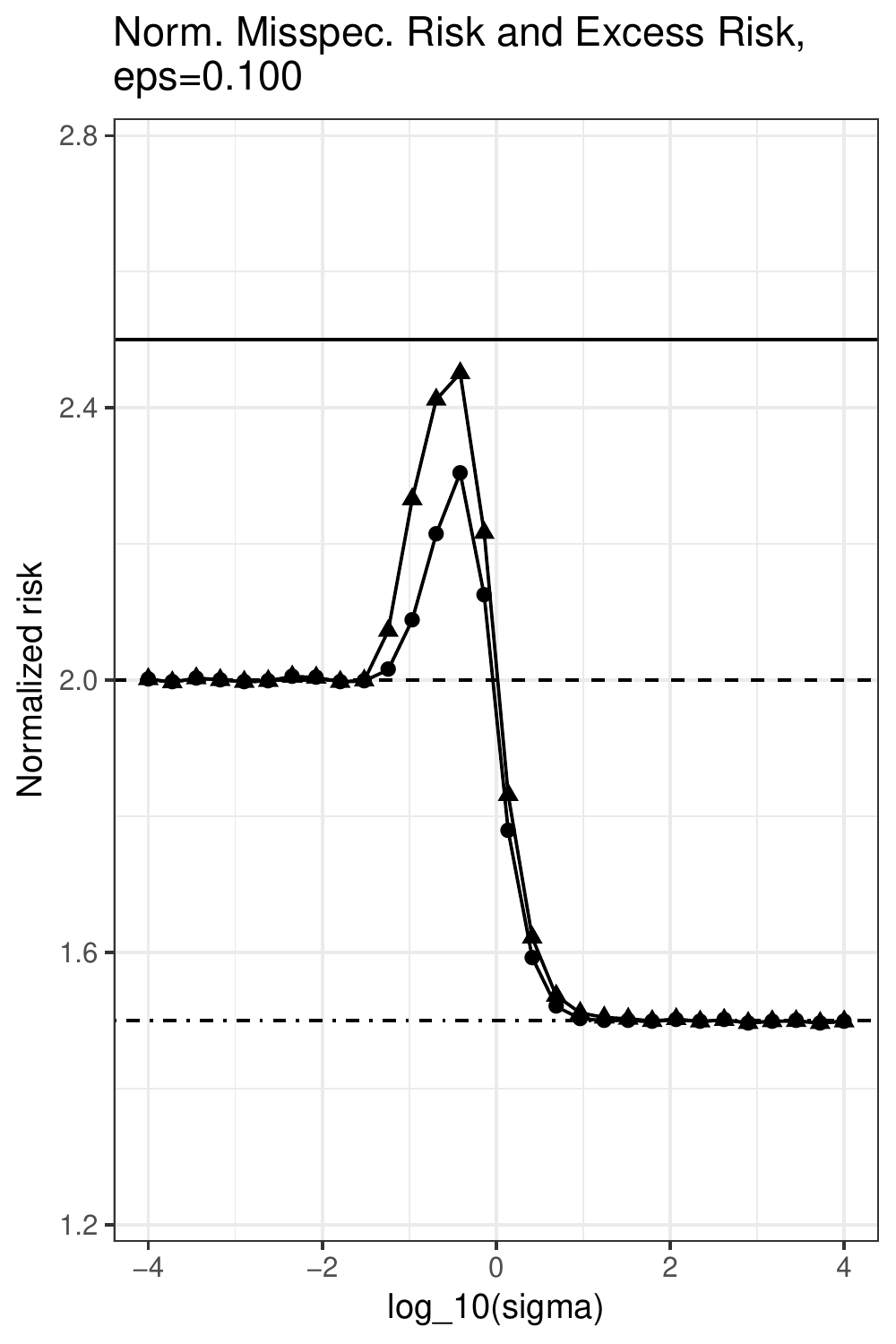}\;
\includegraphics[width=0.3\textwidth]{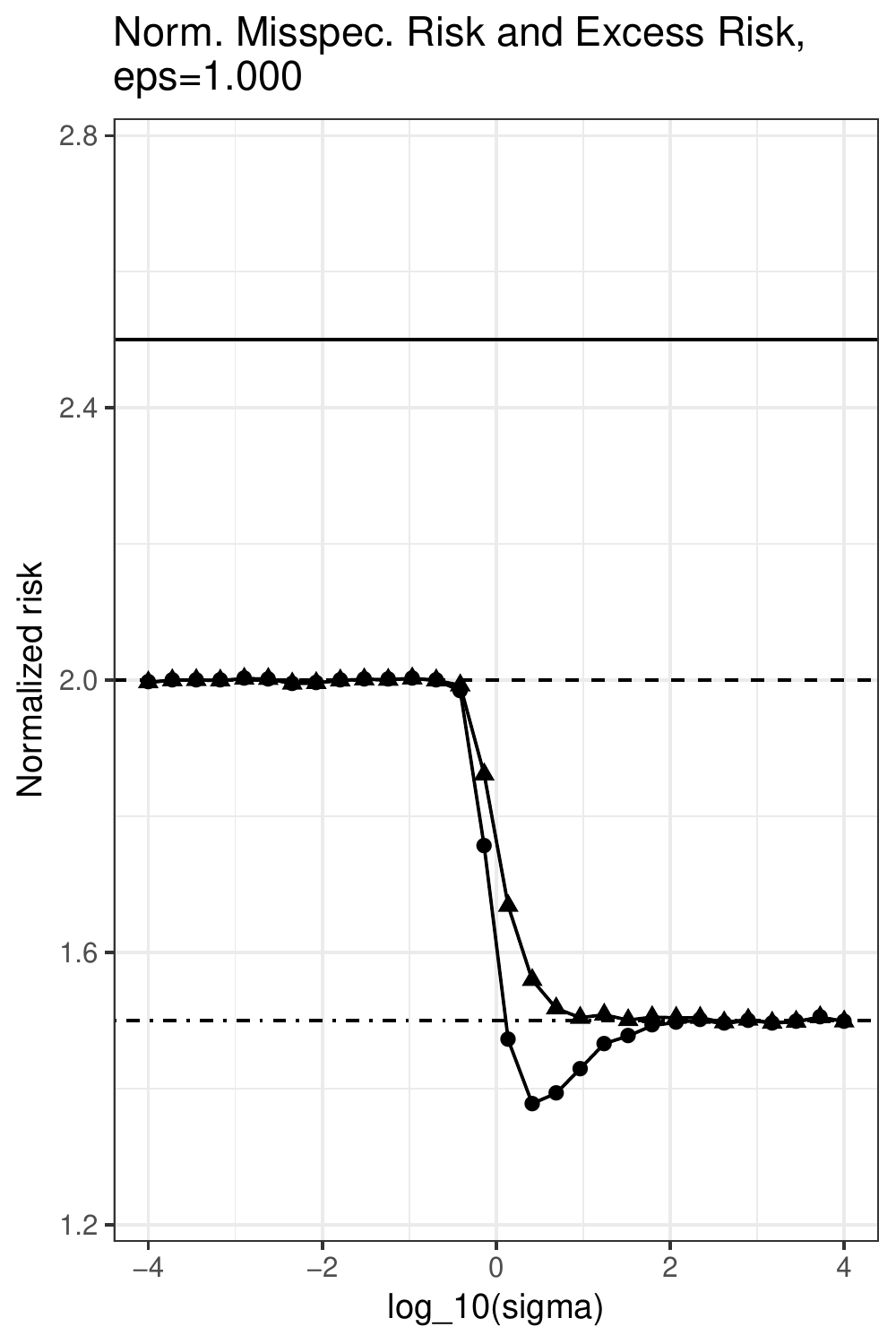}
\caption{
    Empirical estimates of the normalized misspecified risk ($\bullet$)
    and normalized excess risk ($\blacktriangle$) plotted against $\log_{10} (\sigma)$, for the orthant $\C \coloneqq \R^3_+$ and $\theta^*=(1,1,-\epsilon)$ with $\epsilon \in \{0.01, 0.1, 1\}$.
    The solid horizontal line represents the upper bound
    $\delta(T_{\C}(\Pi_\C(\theta^*)))=n-\frac{1}{2}$ guaranteed by
    \eqref{equation:bellec_ub}.
    The dashed line is the common low $\sigma$ limit $n-1$ (see \autoref{corollary:orthant}).
    The dash-dot line is the high $\sigma$ limit $\delta(\R^n_+)=3/2$.
}
\label{figure:sigma_trace_orthant}
\end{figure}

Of course, if one can show that either the low $\sigma$ limit or the high $\sigma$ limit is equal to the upper bound $\delta(T_\C(\Pi_\C(\theta^*)))$, then we know the upper bound is attained either as $\sigma \downarrow 0$ or $\sigma \to \infty$ respectively.
However, in the settings of \autoref{theorem:lb_polyhedron} and \autoref{proposition:high_sigma_limit}, this seldom happens. As discussed already, if $\C$ is polyhedral with nonempty interior, then the low $\sigma$ limit is strictly smaller than the upper bound. If \autoref{proposition:high_sigma_limit} applies, then $K_\C = \bigcap_{\theta \in \C} T_\C(\theta) \subseteq T_\C(\Pi_\C(\theta^*))$ shows that the high $\sigma$ limit is typically strictly smaller than the upper bound; for the special case where $\C$ is a cone, see \autoref{remark:cone_case} for a necessary and sufficient condition for the high $\sigma$ limit to equal the upper bound.

Thus in most cases the suprema are attained at some moderate values of $\sigma$,
but it is difficult to provide a characterization of these maximizing values $\sigma$,
as well as the value of the suprema and whether they are close to the upper bound or not.
The plots suggest that as $\theta^*$ gets closer to $\C$, the suprema get closer to the upper bound as well.

\appendix

\section{Proofs of lemmas in \autoref{section:lb_polyhedron}}
\label{section:lemmas_lb_polyhedron}

The next lemma is a technical device for representing the largest face of a polyhedral cone that lies in a particular hyperplane. It is useful for proving \autoref{lemma:hyperplane_intersection_face} and \autoref{lemma:project_hyperplane}.

\begin{lemma}[Largest face in hyperplane]\label{lemma:largest_face}
Let $\K = \{u : Au \le 0\} \subseteq \R^n$ be a polyhedral cone, where
$A \in \R^{m \times n}$ has distinct rows. For each $y \in \R^n$,
consider the subsets $J \subseteq \{1,\ldots,m\}$ satisfying
\begin{equation}\label{equation:indices_condition}
\{u : A_J u = 0\} \subseteq (y - \Pi_\K(y))^\perp.
\end{equation}
We let $J_y$ denote the smallest such subset.

This subset $J_y$ characterizes a face of $\K$ in the following way.
\begin{equation}\label{equation:largest_face2}
\K \cap (y-\Pi_\K(y))^\perp
=
\{u:A_{J_y} u=0, A_{J_y^c}u \le 0\}.
\end{equation}
\end{lemma}

\begin{proof}
The optimality condition for a projection onto a cone
\eqref{equation:optimality_condition_cone}
implies $\inner{y - \Pi_\K(y), u} \le 0$ for all $u \in \K$.
If $\K$ contains both $u$ and $-u$, then this implies $u \in (y-\Pi_\K(y))^\perp$. Thus for $J=\{1,\ldots,m\}$, \eqref{equation:indices_condition} holds because $\{u : A_J u = 0\} \subseteq \K$. This shows the existence of subsets $J$ that satisfy \eqref{equation:indices_condition}.

Next, note that if $J$ and $J'$ both satisfy \eqref{equation:indices_condition},
then $J\cap J'$ does as well, because
\begin{equation}
\{u : A_{J \cap J'} u = 0\}
= \{u+v : A_Ju = A_{J'}v=0\}
\subseteq (y - \Pi_\K(y))^\perp.
\end{equation}
So, letting $J_y$ be the intersection of all $J$ satisfying \eqref{equation:indices_condition}
yields the unique subset of minimal size.

The $\supseteq$ inclusion in \eqref{equation:largest_face2} follows immediately from $\{u : A_{J_y} u = 0\} \subseteq (y - \Pi_\K(y))^\perp$.
For the other inclusion, suppose $v \in \K \cap (y-\Pi_\K(y))^\perp$.
Then $Av \le 0$, so it remains to verify $A_{J_y}v=0$.
That is, if $J \subseteq \{1,\ldots,m\}$ denotes the indices $j$ for which $\inner{a_j,v}=0$,
we want to show $J_y \subseteq J$; furthermore, this reduces to showing $J$ satisfies \eqref{equation:indices_condition}, by minimality of $J_y$.

Any $u$ satisfying $A_J u = 0$ can be rewritten as $u=v+w$ for some $w$ also satisfying $A_J w =0$. There exists some $c>0$ such that both $v+cw$ and $v-cw$ are in $\K$ because all the linear constraints outside of $J$ are strict inequalities at $v$.
Then, the optimality condition for the projection onto a cone, yields
$\inner{v + cw, y-\Pi_\K(y)} \le 0$
and $\inner{v - cw, y-\Pi_\K(y)} \le 0$.
Since $v \in (y-\Pi_\K(y))^\perp$, this yields $w \in (y-\Pi_\K(y))^\perp$ and thus $u \in (y - \Pi_\K(y))^\perp$, which verifies that $J$ satisfies \eqref{equation:indices_condition}.
\end{proof}

\begin{proof}[Proof of \autoref{lemma:drop_constraints}]
By definition there exist an integer $m$, matrix $A \in \R^{m \times n}$, and vector
$b \in \R^m$ such that $\C \coloneqq \{u \in \R^n : Au \le b\}$.
Fix $\theta^* \in \R^n$ and let $\theta_0 \coloneqq \Pi_\C(\theta^*)$.
We will show
\begin{equation}
T_\C(\theta_0) = \{u : A_J u \le 0\},
\end{equation}
where $J = \{j : \inner{a_j, \theta_0} = b_j\}$.
Then $T_\C(\theta_0)$ is a polyhedral cone.

If $u \in T_\C(\theta_0)$ then for some $r^*>0$ we have $\theta_0 + r u \in \C$.
Thus,
$b_J \ge A_J(\theta_0 + r u) = b_J + r A_J u$
which implies $A_J u \le 0$.

Conversely, suppose $u$ satisfies $A_J u \le 0$. Choose $r^*>0$ so that
$r \inner{a_j, u} \le b_j - \inner{a_j, \theta_0}$ for all $j \notin J$.
This is possible because $b_j > \inner{a_j, \theta_0}$ for each $j \notin J$.
Then $\theta_0 +  r^* u \in \C$ so $u \in T_\C(\theta_0)$.

Finally, we need to prove the second part of the locally polyhedral condition \eqref{split:polyhedral_condition},
which will follow if we show
$T_\C(\theta_0) \cap B_{r^*}(0) \subseteq F_\C(\theta_0)$ for some $r>0$.
If $u \in T_\C(\theta_0)$ then $A_J u \le 0 = b_J - A_J \theta_0$,
so it suffices to find some $r$ such that $A_{J^c} u \le b_{J^c} - A_{J^c} \theta_0$
for any $u \in B_{r^*}(0)$.
For each $j \notin J$, we have $\inner{a_j,\theta_0}< b_j$ so there
exists some $r^*>0$ such that all $\theta \in B_{r^*}(\theta_0)$
satisfy $\inner{a_j,\theta} < b_j$ for all $j \notin J$. Taking
$u=\theta-\theta_0$ concludes the proof.
\end{proof}

\begin{proof}[Proof of \autoref{lemma:hyperplane_intersection_face}]
Let $\T \coloneqq \{u+\Pi_\C(\theta^*):u \in T_\C(\Pi_\C(\theta^*))\}$.
Using the locally polyhedral condition \eqref{split:polyhedral_condition}
and continuity \cite{hiriart2012fundamentals} of $\Pi_\C$ and $\Pi_\T$,
we have
$\Pi_\T(\theta^*)=\Pi_\C(\theta^*)$ (e.g., see the verification of \eqref{equation:projection_trick}),
and thus translating yields
$\Pi_{T_\C(\Pi_\C(\theta^*))}(\theta^*-\Pi_\C(\theta^*))=0$.
Applying \autoref{lemma:largest_face} with $\mathcal{K} = T_\C(\Pi_\C(\theta^*))$, $y=\theta^*-\Pi_\C(\theta^*)$, and $\Pi_\K(y)=0$ concludes the proof.
\end{proof}

\begin{proof}[Proof of \autoref{lemma:jump}]
Fix $\theta_0 \in \C$.
For any $\theta^* \notin \C$ such that $\Pi_\C(\theta^*) = \theta_0$,
\autoref{lemma:drop_constraints} implies the locally polyhedral condition
\eqref{split:polyhedral_condition} holds,
and thus
\autoref{lemma:hyperplane_intersection_face} establishes that
$T_\C(\Pi_\C(\theta^*)) \cap (\theta^* - \Pi_\C(\theta^*))^\perp$
is a face of the tangent cone $T_\C(\theta_0)$.

Since the tangent cone has finitely many faces,
the supremum is actually a maximum over the statistical dimensions of finitely
many such lower-dimensional faces. Thus it remains to show
\begin{equation}
\delta(T_\C(\Pi_\C(\theta^*)) \cap (\theta^* - \Pi_\C(\theta^*))^\perp)
< \delta(T_\C(\theta_0))
\end{equation}
for each $\theta^* \notin \C$ such that $\Pi_\C(\theta^*) = \theta_0$.

The set $(\theta^* - \Pi_\C(\theta^*))^\perp$ is a hyperplane
(not all of $\R^n$) because $\theta^* \notin \C$.
Using the fact that the tangent cone $T_\C(\theta_0)$ has nonempty interior
(because it contains the translation $F_\C(\theta_0)$ of $\C$),
we see that the intersection
$T_\C(\Pi_\C(\theta^*)) \cap (\theta^* - \Pi_\C(\theta^*))^\perp$
is a face that
that lies in a strictly lower-dimensional subspace of $\R^n$,
and is therefore strictly smaller than the full cone $T_\C(\theta_0)$.
Thus, we just need to show $\delta(T') < \delta(T)$
for any polyhedral cone $T$ with nonempty interior in $\R^n$,
and any face $T'$ of $T$ that lies in a strictly lower-dimensional subspace of $\R^n$.

For a point $x \in \R^n$ and a set $S \subseteq \R^n$ let $d(x, S) \coloneqq \inf_{\theta \in S} \norm{x - \theta}$.
Note that the Moreau decomposition for cones \cite[Sec. B]{amelunxen2014living} implies $\norm{\Pi_\K(x)} = d(x, \K^\circ)$ for any $x \in \R^n$ and any cone $\K$,
where $\K^\circ \coloneqq \{u \in \R^n : \inner{u,\theta} \le 0, \forall \theta \in \K\}$
denotes the polar cone of $\K$.
Since $T^\circ \subseteq (T')^\circ$,
we have
\begin{equation}
d(x, (T')^\circ) \le d(x, T^\circ),
\qquad \forall x \in \R^n.
\end{equation}
Thus, if
we show the random vector $Z$ has nonzero probability of being in the set
\begin{equation}
\mathcal{A} \coloneqq \{x \in \R^n : d(x, (T')^\circ) < d(x, T^\circ)\}
= \{x \in \R^n : \norm{\Pi_{T'}(x)} < \norm{\Pi_T(x)}\},
\end{equation}
then we immediately have the desired strict inequality
\begin{equation}
\delta(T') = \E d(Z, (T')^\circ) < \E d(Z, T^\circ) = \delta(T).
\end{equation}
To prove the above claim that $\P(Z \in \mathcal{A}) > 0$, we show below that the interior of $T$ is contained in $\mathcal{A}$;
then our assumption on $Z$ will conclude the proof.

Let $x$ be in the interior of $T$. Then $x \in T \setminus T'$.
Moreover, if we let $U$ be the smallest linear subspace of $\R^n$ containing $T'$,
then $x \notin U$ as well.
Note the the Pythagorean theorem implies
\begin{equation}\label{equation:strict_inequality}
\norm{\Pi_T(x)}^2
= \norm{x}^2
= \norm{\Pi_U(x)}^2 + \norm{x - \Pi_U(x)}^2
> \norm{\Pi_U(x)}^2.
\end{equation}
We also have
\begin{equation}
\Pi_{T'}(x)
= \argmin_{\theta \in T'} \norm{\theta - x}^2
= \argmin_{\theta \in T'} \braces*{
    \norm{\theta - \Pi_U(x)}^2 + \norm{\Pi_U(x) - x}^2
}
= \Pi_{T'}(\Pi_U(x)),
\end{equation}
so combining this with the above inequality \eqref{equation:strict_inequality}
and the optimality condition \eqref{equation:optimality_condition_cone} for
the projection of $\Pi_U(x)$ onto the cone $T'$, we have
\begin{equation}
\norm{\Pi_{T'}(x)}^2
= \norm{\Pi_{T'}(\Pi_U(x))}^2
= \norm{\Pi_U(x)}^2 - \norm{\Pi_U(x) - \Pi_{T'}(\Pi_U(x))}^2
\le \norm{\Pi_U(x)}^2
< \norm{\Pi_T(x)}^2,
\end{equation}
and thus $x \in \mathcal{A}$.
\end{proof}

\begin{proof}[Proof of \autoref{lemma:project_hyperplane}]
The lemma holds immediately if $\theta^* \in \T$, so we assume $\theta^* \notin \T$.

By translating, we may without loss of generality assume $\Pi_\T(\theta^*)=0$
so that the cone is centered at $0$ and can be written as
$\T = \{u : Au \le 0\}$ for some number of constraints $m$
and some matrix $A \in \R^{m \times n}$.
The objective then reduces to
\begin{equation}
    \Pi_\T(y) \in (\theta^*)^\perp,
    \qt{for all $y \in B_r(\theta^*)$.}
\end{equation}

For any $y \in \R^n$ let $J_y \subseteq \{1,\ldots,m\}$ be as defined in \autoref{lemma:largest_face} for our polyhedral cone $\T$; it characterizes the largest face of $\T$ that lies in $(\theta^*)^\perp$.
We claim there exists $r>0$ such that
\begin{equation}\label{equation:small_ball_claim}
\{u : A_{J_y} u = 0\} \subseteq (\theta^*)^\perp,
\quad \forall y \in B_r(\theta^*).
\end{equation}
If not, then there exists a sequence of points $y_k \notin \T$ converging to $\theta^*$
such that $\{u : A_{J_{y_k}} u = 0\} \not\subseteq (\theta^*)^\perp$
for all $k$.
Since there are finitely many distinct subsets $J_{y_k}$, we may take a subsequence
and without loss of generality assume it is common subset $J=J_{y_k}$ for all $k$,
and $\{u:A_J u = 0\} \not\subseteq (\theta^*)^\perp$.
By the definition \eqref{equation:indices_condition} of $J_{y_k}$,
any $u$ satisfying $A_J u = 0$ also satisfies $\inner{y_k - \Pi_\T(y_k), u}=0$.
By continuity of $\Pi_\T$ and taking $k \to \infty$, we have $\inner{\theta^*,u} = 0$ as well,
a contradiction.

Finally, since the optimality condition \eqref{equation:optimality condition}
for $\Pi_\T$ implies
$\inner{\Pi_\T(y),y-\Pi_\T(y)}=0$ for any $y \in \R^n$,
\eqref{equation:largest_face2} implies $\Pi_\T(y) \in \{u : A_{J_y}u=0\}$.
Combining this with \eqref{equation:small_ball_claim}
concludes the proof.
\end{proof}

\section{Proofs for \autoref{subsection:isotonic} (isotonic regression)}
\label{section:proofs_isotonic}

\subsection{Proofs of block monotone cone lemmas}
\label{subsection:proofs_block_monotone}

\begin{proof}[Proof of \autoref{lemma:weighted_isotonic_regression}]
The first claim follows from decomposing the squared Euclidean distance into blocks.
\begin{align}
\min_{v \in \S_{\abs{I_1},\ldots,\abs{I_m}}}
\norm{v-z}^2
&= \min_{x \in \S^m} \sum_{j=1}^m \sum_{i \in I_j} (x_j -z_i)^2\\
&= \min_{x \in \S^m} \sum_{j=1}^m \sum_{i \in I_j} ((x_j -\bar{z}_{I_j})^2 + (\bar{z}_{I_j} - y_i)^2)\\
&= \sum_{j=1}^m \sum_{i \in I_j} (z_i-\bar{z}_{I_j})^2+\min_{x \in \S^m} \sum_{j=1}^m \abs{I_j} (x_j-\bar{z}_{I_j})^2.
\end{align}

Let $Z$ and $Z'$ be standard Gaussian in $\R^n$ and $\R^m$ respectively.
If $\abs{I_1}=\cdots=\abs{I_m}=r$, then the first claim implies
\begin{align}
\delta(\S_{\abs{I_1},\ldots,\abs{I_m}})
\coloneqq \E \norm{\Pi_{\S_{\abs{I_1},\ldots,\abs{I_m}}}(Z)}^2
\overset{(i)}{=} r \E \norm{\Pi_{\S^m}(Z'/\sqrt{r})}^2
\overset{(ii)}{=} \E \norm{\Pi_{\S^m}(Z')}^2
\eqqcolon \delta(\S^m)
=\sum_{j=1}^m \frac{1}{j},
\end{align}
where (i) is due to $Z'/\sqrt{r} \overset{d}{=} (\bar{Z}_{I_1},\ldots,\bar{Z}_{I_m})$,
and (ii) is due to $\Pi_\C(cx)=c\Pi_\C(x)$ for a cone $\C$ and $c>0$ (e.g., \cite[Sec. 1.6]{bellec2015sharpshape}).
The statistical dimension of $\S^m$ is proved by \citet[Sec. D.4]{amelunxen2014living}.
\end{proof}

\begin{proof}[Proof of \autoref{lemma:block_monotone_stat_dim}]
We use two useful properties of the statistical dimension of any cone $\C$ \cite[Prop. 3.1]{amelunxen2014living}.
\begin{itemize}
    \item Rotational invariance: for any orthogonal transformation $Q$, we have $\delta(Q\C) = \delta(\C)$.
    \item Invariance under embedding: $\delta(\C \times \{0\}^k) = \delta(\C)$.
\end{itemize}
Thus it suffices to provide an orthogonal transformation $Q$ such that $Q\S_{\abs{I_1},\ldots,\abs{I_m}}$
is an embedding of the cone \eqref{equation:block_monotone_embed} into $\R^n$.

Let $e_i$ denote the $i$th standard basis vector in $\R^n$.
Let the last element of each block be denoted $k_j \coloneqq \max I_j$ for $1 \le j \le m$,
with $k_0=0$ for convenience.
The block monotone cone $\S_{\abs{I_1},\ldots,\abs{I_m}}$ is defined by the following constraints for $u \in \R^n$.
\begin{subequations}
\begin{align}
\inner{e_i-e_{i+1}, u}
&\le 0,
& i\in \{k_1,\ldots,k_m\}
\label{align:inequality_constraint}
\\
\inner{e_i-e_{i+1}, u}
&= 0,
& i \in \{1,\ldots,n-1\} \setminus \{k_1,\ldots,k_m\}
\label{align:equality_constraint}
\end{align}
\end{subequations}

Let us focus on an arbitrary block $I_j$.
Consider the $\abs{I_j} \times \abs{I_j}$ matrix
\begin{equation}
\tilde{A}_j
=
\begin{bmatrix}
1 & &
\\
-1 & 1 &
\\
& -1 & 1
\\
 &  & \ddots & \ddots\\
& & & -1 & 1
\end{bmatrix}
\end{equation}
Because $\tilde{A}_j$ is full rank,
the QR decomposition implies there exists an $\abs{I_j} \times \abs{I_j}$ orthogonal matrix $\tilde{Q}_j$ such that $\tilde{R}_j\coloneqq\tilde{Q}_j \tilde{A}_j$ is upper triangular with positive diagonal entries, and this decomposition is unique.

The block diagonal matrix $Q$ with blocks $\tilde{Q}_1,\ldots,\tilde{Q}_m$
is an $n \times n$ orthogonal matrix.
Let $A$ and $R$ also be block diagonal, each constructed similarly using the $\tilde{A}_j$ and the $\tilde{R}_j$ respectively, so that $U=QA$.
We consider $Q \S_{\abs{I_1},\ldots,\abs{I_m}}$.
We use the fact that if $v=Qu$ then $\inner{b,u} \le 0 \iff \inner{Qb, v} \le 0$
to rewrite the constraints \eqref{align:inequality_constraint}
and \eqref{align:equality_constraint}.
The following hold for each $j=1,\ldots,m$.
\begin{itemize}
    \item
    Note that the $i$th column of $A$ is $a_i=e_i-e_{i+1}$
    when $k_{j-1} < i < k_j$.
    For these $i$, the equality constraints \eqref{align:equality_constraint}
    after the transformation become $0 = \inner{Q(e_i-e_{i+1}),v} = \inner{r_i,v}$ where $r_i$ is the $i$th column of $R$.
    Since $\tilde{R}_j$ is upper triangular with nonzero diagonal entries (because $\tilde{A}_j$ is full rank),
    induction on $i=k_{j-1}+1,\ldots,k_j - 1$ implies
    \begin{equation}
    v_i = 0, \qquad k_{j-1} < i < k_j.
    \end{equation}
    \item
    When $j<m$, we have $e_{k_j}=a_k$
    and $e_{k_j+1} = a_{k_j} + a_{k_j+1} + \cdots + a_{k_{j+1}}$
    Thus for $j<m$ the inequality constraint $\inner{e_{k_j}-e_{k_j+1},u} \le 0$ becomes
    \begin{equation}
    0 \ge \inner{Q(e_{k_j} - e_{k_j+1}), v}
    = \inner{r_{k_j} - r_{k_j+1} - r_{k_j+2} - \cdots - r_{k_{j+1}}, v}
    = \inner{r_{k_j} - r_{k_{j+1}}, v},
    \end{equation}
    where the last equality is due to $\inner{r_i,v}=0$ for $k_j < i < k_{j+1}$, by the previous point.
    Since $\tilde{R}_j$ and $\tilde{R}_{j+1}$ are each upper triangular,
    the inequality reduces to
    $r_{k_j,k_j} v_{k_j} \le r_{k_{j+1}, k_{j+1}} v_{k_{j+1}}$,
    where $r_{k,k}$ denotes the $k$th diagonal entry of $R$.
    \autoref{lemma:qr} (proved below) computes these diagonal elements and yields
    \begin{equation}
    \frac{v_{k_j}}{\sqrt{\abs{I_j}}}  \le \frac{v_{k_{j+1}}}{\sqrt{\abs{I_{j+1}}}}.
    \end{equation}
\end{itemize}

Therefore we have shown that
$Q \S_{\abs{I_1},\ldots,\abs{I_m}}$
consists of all vectors satisfying
\begin{equation}
\frac{v_{k_1}}{\sqrt{\abs{I_1}}} \le \frac{v_{k_2}}{\sqrt{\abs{I_2}}} \le \cdots \le \frac{v_{k_m}}{\sqrt{\abs{I_m}}},
\text{ and }
v_i=0, \forall i \in \{1,\ldots,n\} \setminus \{k_1,\ldots,k_m\}.
\end{equation}
We have thus verified the claim that $Q\S_{\abs{I_1},\ldots,\abs{I_m}}$
is an embedding of \eqref{equation:block_monotone_embed} into $\R^n$.

When the blocks all have equal size $r$, the cone \eqref{equation:block_monotone_embed} becomes the monotone cone $\S^m$, whose statistical dimension is $\sum_{j=1}^m \frac{1}{j}$ \cite[Sec. D.4]{amelunxen2014living}.
\end{proof}

\begin{lemma}\label{lemma:qr}
Consider the $n \times n$ matrix
\begin{equation}
A=
\begin{bmatrix}
1 & &
\\
-1 & 1 &
\\
& -1 & 1
\\
 &  & \ddots & \ddots\\
& & & -1 & 1
\end{bmatrix}.
\end{equation}
There exists a unique orthogonal matrix $Q$ and a unique upper triangular matrix $R$ with positive diagonal entries such that $A=QR$.
The bottom-right entry of $R$ is $r_{n,n} = 1/\sqrt{n}$.
\end{lemma}

\begin{proof}
Let $q_i$ be the $i$th column of $Q$.
The last column $q_n$ is orthogonal to the span of the first $n-1$ columns of $A$,
so $q_n$ is either $(1,\ldots,1)/\sqrt{n}$ or its negative. The positivity constraint on the diagonal entries of $R$ implies the former, and thus $r_{n,n} = \inner{q_n, e_n} = 1/\sqrt{n}$.
\end{proof}

\subsection{Statistical dimension of the block monotone cone in general}
\label{subsection:stat_dim_block_monotone}

In \autoref{lemma:block_monotone_stat_dim} we provided an expression
for the statistical dimension of the block monotone cone $\S_{\abs{I_1},\ldots,\abs{I_m}}$ when the block sizes were equal.
In general, the statistical dimension can be higher or lower than $\sum_{j=1}^m \frac{1}{j}$.
Consider the following examples for $m=3$.

\autoref{lemma:block_monotone_stat_dim} implies
$\S_{n-2,1,1}$ has the same statistical dimension as
$\{v \in \R^3 : v_1/\sqrt{n-2} \le v_2\le v_3\}$.
As $n \to \infty$ this latter cone approaches $\{v \in \R^3 : 0 \le v_2 \le v_3\}$ which has statistical dimension $1+\parens*{\frac{1}{8} \cdot 2 + \frac{1}{2} \cdot 1}=\frac{7}{4}=1.75$,
which is smaller than $\sum_{j=1}^3 \frac{1}{j}=\frac{11}{6} = 1.8\bar{3}$.

On the other hand, $\S_{1,n-2,1}$ has the same statistical dimension as $\{v \in \R^3 : v_1 \le v_2/\sqrt{n-2} \le v_3\}$. As $n \to \infty$ this latter cone approaches
$\{v \in \R^3 : v_1 \le 0, v_3 \ge 0\}$ which has statistical dimension $1+\frac{1}{2} + \frac{1}{2} = 2$, which is larger than $1.8\bar{3}$.

We suspect that the approach used to prove the statistical dimension of $\S^n$ \cite[Sec. D.4]{amelunxen2014living},
which uses the theory of finite reflection groups,
cannot be generalized for $\S_{\abs{I_1},\ldots,\abs{I_m}}$, due to the asymmetry of \eqref{equation:block_monotone_embed}.
However, using a result of \citet{klivans2011projection}, it is possible to show that the average statistical dimension among all block monotone cones with a given [unordered] set of $m$ block sizes is $H_m$ \cite[Prop. 6.6]{amelunxen2015intrinsic}.

\subsection{Proof of \autoref{proposition:isotonic}}
\label{subsection:proof_isotonic}

When applying \autoref{theorem:lb_polyhedron}, it is useful to characterize $\S^n$ and its tangent cones using conic generators. If $T \subseteq \R^n$ is a cone and there exist $x_1,\ldots,x_p \in T$ such that
\begin{equation}
    T = \braces*{
        \sum_{i=1}^p \alpha_i x_i:
        \alpha_i \ge 0, \forall i
    },
\end{equation}
then we call $x_1,\ldots,x_p$ the \textit{conic generators} of $T$, and write
\begin{equation}
    T = \op{cone}\{x_1,\ldots,x_p\}.
\end{equation}

\begin{lemma}\label{lemma:intersection_cone_gen}
Let $\theta^* \in \R^n$ and let $\C \subseteq \R^n$ be closed and convex.
If the tangent cone $T_\C(\Pi_\C(\theta^*))$ is generated by $x_1,\ldots,x_p \in \R^n$,
i.e. $T_\C(\Pi_\C(\theta^*)) = \op{cone}\{x_1,\ldots,x_p\}$,
then
\begin{equation}
    T_\C(\Pi_\C(\theta^*)) \cap (\theta^*-\Pi_\C(\theta^*))^\perp
    = \op{cone}(\{x_1,\ldots,x_p\} \cap (\theta^*-\Pi_\C(\theta^*))^\perp).
\end{equation}
\end{lemma}

\begin{proof}[Proof of \autoref{lemma:intersection_cone_gen}]
The inclusion $\supset$ is immediate, so it remains to prove the inclusion $\subseteq$.
Note that the optimality condition \eqref{equation:optimality condition} implies
$\inner{\theta^* - \Pi_\C(\theta^*), x} \le 0$ for any $x \in T_\C(\Pi_\C(\theta^*))$.
In particular, if $v \in T_\C(\Pi_\C(\theta^*)) \cap (\theta^*-\Pi_\C(\theta^*))^\perp$,
then $v$ can be written as the conical combination $v=\sum_{i=1}^p \alpha_i x_i$ with $\alpha_i \ge 0$,
and we have
\begin{equation}
    0
    = \inner{\theta^*-\Pi_\C(\theta^*), v}
    = \sum_{i=1}^p \alpha_i \underbrace{\inner{\theta^*-\Pi_\C(\theta^*), x_i}}_{\le 0}.
\end{equation}
Thus, if a generator $x_i$ is not in the hyperplane $(\theta^*-\Pi_\C(\theta^*))^\perp$,
then $\alpha_i=0$, so $x_i$ does not contribute in the conical combination of $v$.
Thus, $v$ can be written as a conical combination
of generators in $(\theta^*-\Pi_\C(\theta^*))^\perp$.
\end{proof}

We are now ready to prove \autoref{proposition:isotonic}.

\begin{proof}[Proof of \autoref{proposition:isotonic}]
By \autoref{theorem:lb_polyhedron},
it suffices to prove that the statistical dimension term is
$\sum_{k=1}^K \delta\parens*{\S_{\abs{I_1^k},\ldots, \abs{I^k_{m_k}}}}$.

For $p \ge 1$ let
\begin{equation}
    M_p \coloneqq
    \begin{bmatrix}
        -1 & -1 & \cdots & -1
        \\
        1 & 1 & \cdots & 1
        \\
        & 1 & \cdots & 1
        \\
        & & \ddots & \vdots
        \\
        & & & 1
    \end{bmatrix}
    \in \R^{(p+1) \times p}.
\end{equation}
The rows of $M_p$ are the conic generators of $\S^p$.

Suppose first that $\Pi_{\S^n}(\theta^*)$ is constant, so that $K=1$ and $J_1=\{1,\ldots,n\}$.
Then
$\Pi_{\S^n}(\theta^*) = (\mu_1,\mu_1,\ldots,\mu_1)$
where $\mu_1\coloneqq \frac{1}{n} \sum_{i=1}^n \theta^*_i$;
this follows directly by minimizing $\sum_{i=1}^n (\theta_i^* - \mu_1)^2$
with respect to $\mu_1$.

The finest partition $(I^1_1,\ldots,I^1_{m_1})$ of $J_1$ into blocks satisfying
\eqref{equation:partition_mean} can be constructed greedily as follows.
Begin populating $I^1_1$ with the elements of $\{1,\ldots,n\}$ in order,
stopping as soon as the mean of the elements of $I^1_1$ is $\mu_1$.
Then begin populating $I^1_2$ with the remaining elements in order,
again stopping when the mean of the elements in $I^1_2$ is $\mu_1$.
Continue in this manner until
the last element $n$ is placed in a subset $I^1_{m_1}$.
The mean of the elements of this last subset $I^1_{m_1}$ is $\mu_1$ as well,
since the mean of all components of $\theta^*$ is $\mu_1$.
Thus this partition satisfies \eqref{equation:partition_mean}.
To establish uniqueness, note that if some other partition of $J_1$ satisfies \eqref{equation:partition_mean},
then our partition $(I^1_1,\ldots,I^1_{m_1})$ must be a refinement,
due to the greedy construction.

Because $\Pi_{\S^n}(\theta^*)$ is constant,
the tangent cone there is
$T_{\S^n}(\Pi_{\S^n}(\theta^*))=\S^n$
\cite[Prop. 3.1]{bellec2015sharpshape}, which is generated by the rows of $M_n$.
In order to use \autoref{lemma:intersection_cone_gen},
we need to determine which rows of $M_n$ are in the hyperplane
$(\theta^* - \Pi_{\S^n}(\theta^*))^\perp$.
We already know the mean of the components of
$\theta^* - \Pi_{\S^n}(\theta^*)$
is zero, so the first two rows are in the hyperplane.

We claim that exactly $m_1-1$ of the remaining
$n-1$ rows of $M_n$ also lie in the hyperplane.
Explicitly, if $(I^1_1,\ldots,I^1_{m_1})$
is without loss of generality assumed to be sorted in increasing order,
then the remaining rows of $M_n$ that lie in the hyperplane are
the indicator vectors for
\begin{equation}\label{equation:subsets}
    \bigcup_{j= u}^{m_k} I^1_j,
    \qquad 2 \le u \le m_k.
\end{equation}
No other rows of $M_n$ can be in the hyperplane,
else there would exist a finer partition of $J_1$.

So, \autoref{lemma:intersection_cone_gen} implies
$T_{\S^n}(\Pi_{\S^n}(\theta^*))
\cap (\theta^*-\Pi_{\S^n}(\theta^*))^\perp$
is the cone generated by $(-1,\ldots,-1)$, $(1,\ldots,1)$,
and the indicator vectors of the subsets \eqref{equation:subsets},
otherwise known as the cone of nondecreasing vectors that are
piecewise constant on the blocks $I^1_1,\ldots,I^1_{m_1}$.
Its statistical dimension is denoted by $\delta(\S_{\abs{I^1_1},\ldots,\abs{I^1_{m_1}}})$.
This concludes the proof in the case when $\Pi_{\S^n}(\theta^*)$ is constant.

We now turn to the general case where
$\Pi_{\S^n}(\theta^*)$ is piecewise constant
with values $\mu_1 < \cdots < \mu_K$
on $J_1,\ldots,J_K$ respectively.
We claim
\begin{equation}\label{equation:monotone_proj}
\mu_k = \frac{1}{\abs{J_k}} \sum_{i \in J_k} \theta^*_i.
\end{equation}
Since $\S^n$ is a cone, the projection satisfies
$\inner{\theta^* - \Pi_{\S^n}(\theta^*), x} \le 0$ for all $x \in \S^n$, with equality if $x=\Pi_\C(\theta^*)$ (e.g., \cite[Sec. 1.6]{bellec2015sharpshape}).
Letting $x_1,\ldots,x_{n+1}$ be the conic generators of $\S^n$ (the rows of $M_n$), we have $\Pi_\C(\theta^*) = \sum_{i=1}^{n+1} \alpha_i x_i$ for some coefficients $\alpha_i \ge 0$. Then,
\begin{equation}
    0 = \inner{\theta^* - \Pi_{\S^n}(\theta^*), \Pi_\C(\theta^*)} = \sum_{i=1}^{n+1} \alpha_i \underbrace{\inner{\theta^* - \Pi_{\S^n}(\theta^*), x_i}}_{\le 0},
\end{equation}
which implies $\inner{\theta^* - \Pi_{\S^n}(\theta^*), x_i}=0$ if $\alpha_i > 0$.
Consequently, if $\Pi_{\S^n}(\theta^*)$ changes value from component $j-1$ to $j$, then $\sum_{i=j}^n [\theta^*_i - (\Pi_{\S^n}(\theta^*))_i] = 0$. Thus \eqref{equation:monotone_proj} holds.

By Proposition 3.1 of \cite{bellec2015sharpshape},
the tangent cone is
\begin{equation}
    T_{\S^n}(\Pi_{\S^n}(\theta^*))
    = \S^{n_1} \times \cdots \times \S^{n_K},
\end{equation}
which is generated by the rows of the block diagonal matrix
\begin{equation}
    A\coloneqq
    \begin{bmatrix}
        M_{n_1}\\
        & \ddots\\
        && M_{n_K}
    \end{bmatrix}.
\end{equation}
To find which rows of $A$ are in the hyperplane
$(\theta^* - \Pi_{\S^n}(\theta^*))^\perp$,
we can treat each block $M_{n_k}$ separately
and repeat the above argument.
Doing so shows that
$T_{\S^n}(\Pi_{\S^n}(\theta^*))
\cap (\theta^*-\Pi_{\S^n}(\theta^*))^\perp$
is the cone of vectors that are piecewise constant on
$(I^1_1,\ldots,I^1_{m_1},\ldots,I^K_1, \ldots, I^K_{m_K})$
and are increasing within each of the blocks  $(J_1,\ldots,J_K)$.
The statistical dimension of this cone is $\sum_{k=1}^K \delta(\S_{\abs{I^k_1},\ldots,\abs{I^k_{m_k}}})$.
\end{proof}

\section{Proof of \autoref{proposition:ball}}
\label{section:ball_proof}

Let $r\coloneqq \norm{\theta^*}$.
By rotating the problem, we may without loss of generality assume $\theta^* = (r,0,\ldots,0)$.

Let $E \coloneqq \{Y \in B_{(r-1)/2}(\theta^*)\}$.
Then we have $E \subseteq \{Y \notin \C\}$,
so under the event $E$ we have $\hat{\theta}(Y) = Y/\norm{Y}$.
Noting $\norm{Y}^2=\norm{\theta^*+\sigma Z}^2 = r^2 + 2 \sigma r Z_1 + \sigma^2 \norm{Z}^2$, we have
\begin{equation}
\frac{1}{\sigma^2}\norm{\hat{\theta}(Y) - \Pi_\C(\theta^*)}^2
= \frac{1}{\sigma^2} \parens*{
    \frac{r + \sigma Z_1}{\sqrt{r^2 + 2 \sigma r Z_1 + \sigma^2 \norm{Z}^2}}
    - 1
}^2
+ \frac{\sum_{i=2}^n Z_i^2}{r^2 + 2 \sigma r Z_1 + \sigma^2 \norm{Z}^2}.
\end{equation}
The second term converges to $r^{-2}\sum_{i=2}^n Z_i^2$ as $\sigma \downarrow 0$.
We show the first term vanishes as $\sigma \downarrow 0$.
Defining $g(\sigma) \coloneqq \norm{\theta^*+\sigma Z}$, we have
\begin{align}
g(\sigma) &= \sqrt{r^2 + 2 \sigma r Z_1 + \sigma^2 \norm{Z}^2}\\
g'(\sigma) &= \frac{r Z_1 + \sigma \norm{Z}^2}{g(\sigma)}\\
g''(\sigma) &= \frac{\norm{Z}^2}{g(\sigma)} - \frac{(rZ_1 + \sigma \norm{Z}^2) g'(\sigma)}{g(\sigma)^2}
\end{align}
Moreover we have $g(0)=r$, $g'(0)= Z_1$, and $g''(0) = (\norm{Z}^2-Z_1^2)/r$.
Then by L'H\^opital's rule,
\begin{align}
&\lim_{\sigma \downarrow 0}
\frac{1}{\sigma} \parens*{
    \frac{r + \sigma Z_1}{\sqrt{r^2 + 2 \sigma r Z_1 + \sigma^2 \norm{Z}^2}}
    - 1
}
\\
&= \lim_{\sigma \downarrow 0}
\frac{r+\sigma Z_1 - g(\sigma)}{\sigma g(\sigma)}
= \lim_{\sigma \downarrow 0}
\frac{Z_1 - g'(\sigma)}{g(\sigma) + \sigma g'(\sigma)}
= \frac{Z_1 - Z_1}{r + 0} = 0.
\end{align}
Note $\b{1}_E \to 1$ almost surely as $\sigma \downarrow 0$.
Thus,
$\sigma^{-2} \norm{\hat{\theta}(Y) - \Pi_\C(\theta^*)}^2 \b{1}_E \to r^{-2} \sum_{i=2}^n Z_i^2$ almost surely.
By the upper bound \eqref{equation:bellec_ub} we may use the dominated convergence theorem to get
\begin{equation}
\lim_{\sigma \downarrow 0}
\frac{1}{\sigma^2} \E_{\theta^*} \brackets*{
    \norm{\hat{\theta}(Y) - \Pi_\C(\theta^*)}^2 \b{1}_E
}
= \frac{1}{r^2} \sum_{i=2}^n \E Z_i^2 = \frac{n-1}{r^2}.
\end{equation}
To conclude the proof of the first limit \eqref{align:ball_limit1}, note that
\begin{equation}
\lim_{\sigma \downarrow 0}
\frac{1}{\sigma^2} \E_{\theta^*} \brackets*{
    \norm{\hat{\theta}(Y) - \Pi_\C(\theta^*)}^2 \b{1}_{E^c}
}
=0,
\end{equation}
which holds by the argument used in the proof of \autoref{theorem:lb_polyhedron} (e.g., see the second term in \eqref{equation:tail}).

A similar proof holds for the second limit \eqref{align:ball_limit2}.
Let $E$ and $g(\sigma)$ be the same as before.
Then
\begin{align}
&\frac{1}{\sigma^2} \parens*{
    \norm{\hat{\theta}(Y)-\theta^*}^2
    - \norm{\Pi_\C(\theta^*)-\theta^*}^2
}
\\
&= \frac{1}{\sigma^2} \parens*{
    \frac{r + \sigma Z_1}{\sqrt{r^2 + 2 \sigma r Z_1 + \sigma^2 \norm{Z}^2}}
    - r
}^2
+ \frac{\sum_{i=2}^n Z_i^2}{r^2 + 2 \sigma r Z_1 + \sigma^2 \norm{Z}^2}
- \frac{(r-1)^2}{\sigma^2}
\\
&=
\frac{1}{\sigma^2}\brackets*{\parens*{\frac{r+\sigma Z_1}{g(\sigma)} - r}^2 - (r-1)^2}
+ \frac{\sum_{i=2}^n Z_i^2}{r^2 + 2 \sigma r Z_1 + \sigma^2 \norm{Z}^2}.
\end{align}
Again, the second term tends to $r^{-2} \sum_{i=2}^n Z_i^2$ as $\sigma \downarrow 0$.
To handle the first term we use L'H\^opital's rule again.
Let
\begin{align}
h(\sigma) &\coloneqq
\frac{r+\sigma Z_1}{g(\sigma)} - r
\\
h'(\sigma) &=
\frac{Z_1}{g(\sigma)} - \frac{(r+\sigma Z_1) g'(\sigma)}{g(\sigma)^2}
\\
h''(\sigma) &=
- \frac{Z_1 g'(\sigma)}{g(\sigma)^2}
+ 2\frac{(r+\sigma Z_1)g'(\sigma)^2}{g(\sigma)^3}
- \frac{Z_1 g'(\sigma) + (r+\sigma Z_1)g''(\sigma)}{g(\sigma)^2}
\end{align}
Recalling the limits $g(0) = r$, $g'(0) = Z_1$, and $g''(0) =(\norm{Z}^2-Z_1^2)/r$,
we have $h(\sigma) \to -(r-1)$, $h'(\sigma) \to 0$, and
\begin{equation}
h''(0) =
- \frac{Z_1^2}{r^2}
+ 2 \frac{rZ_1^2}{r^3}
- \frac{Z_1^2 + \norm{Z}^2 - Z_1^2}{r^2}
= \frac{Z_1^2 - \norm{Z}^2}{r^2}.
\end{equation}

Then, L'H\^opital's rule allows us to compute the limit of the first term.
\begin{align}
&\lim_{\sigma \downarrow 0}
\frac{1}{\sigma^2}\brackets*{\parens*{\frac{r+\sigma Z_1}{g(\sigma)} - r}^2-(r-1)^2}
\\
&=
\lim_{\sigma \downarrow 0}
\frac{h(\sigma)^2 - (r-1)^2}{\sigma^2}
= \lim_{\sigma \downarrow 0}
\frac{h(\sigma) h'(\sigma)}{\sigma}
= \lim_{\sigma \downarrow 0}
(h'(\sigma)^2 + h(\sigma) h''(\sigma))
= \frac{(r-1)(\norm{Z}^2-Z_1^2)}{r^2}.
\end{align}
Combining terms yields
\begin{equation}
\frac{1}{\sigma^2} \parens*{\norm{\hat{\theta}(Y) - \theta^*}^2 - \norm{\Pi_\C(\theta^*)-\theta^*}^2}
\b{1}_E
\to \frac{(r-1)(\norm{Z}^2 -Z_1^2) + \sum_{i=2}^n Z_i^2}{r^2}
= \frac{\sum_{i=2}^n Z_i^2}{r},
\end{equation}
so again by dominated convergence with the upper bound \eqref{equation:bellec_ub}, we have
\begin{equation}
\frac{1}{\sigma^2} \E_{\theta^*}\brackets*{
    \parens*{\norm{\hat{\theta}(Y) - \theta^*}^2 - \norm{\Pi_\C(\theta^*)-\theta^*}^2}
    \b{1}_E
}
\to
\frac{n-1}{r}.
\end{equation}
To conclude the proof of \eqref{align:ball_limit2},
note that
\begin{equation}
\frac{1}{\sigma^2} \E_{\theta^*}\brackets*{
    \parens*{\norm{\hat{\theta}(Y) - \theta^*}^2 - \norm{\Pi_\C(\theta^*)-\theta^*}^2}
    \b{1}_{E^c}
}
\to
0,
\end{equation}
which was proved in the proof of \autoref{theorem:lb_polyhedron} (see \eqref{equation:excess_tail}).

\section{Proofs for \autoref{subsection:high_noise_limit}}
\label{section:proofs_high_noise_limit}

The following lemma shows that the left-hand side of \eqref{equation:boundedness} is nonnegative.

\begin{lemma}\label{lemma:boundedness_nonnegative}
For any $\theta_0 \in \C$,
\begin{equation}
\norm{\Pi_{F_\C(\theta_0)}(x)}^2
\ge \norm{\Pi_{K_\C}(x)}^2.
\end{equation}
\end{lemma}

\begin{proof}[Proof of \autoref{lemma:boundedness_nonnegative}]
Because $K_\C$ is a cone, we have $\inner{x, \Pi_{K_\C}(x)} = \norm{\Pi_{K_\C}(x)}^2$.
Since $K_\C \subseteq F_\C(\theta_0)$, the optimality condition for $\Pi_{F_\C(\theta_0)}(x)$ implies $\inner{x-\Pi_{F_\C(\theta_0)}(x), \Pi_{K_\C}(x)} \le 0$ and thus
\begin{equation}
\norm{\Pi_{K_\C}(x)}^2
\le \inner{\Pi_{F_\C(\theta_0)}(x), \Pi_{K_\C}(x)}
\le \norm{\Pi_{F_\C(\theta_0)}(x)} \norm{\Pi_{K_\C}(x)}.
\end{equation}
Thus $\norm{\Pi_{F_\C(\theta_0)}(x)} \ge \norm{\Pi_{K_\C}(x)}$ and $M_{\theta_0} \ge 0$.
\end{proof}

\begin{proof}[Proof of \autoref{lemma:core_cone}]
We first prove the equalities (i) and (ii).
\begin{enumerate}[(i)]
    \item
    Let $v \in \{u : \R_+ u \subseteq F_\C(\theta_0)\}$ and let $\theta \in \C$.
    For any $c>0$ we have $\theta_0+cv \in \C$, and convexity implies $\theta + \alpha (\theta_0+cv - \theta) \in \C$ for all $\alpha \in [0,1]$.
    For large $c$ we have $\norm{\theta_0 + cv-\theta}>1$ and thus
    $\theta + \frac{\theta_0+cv - \theta}{\norm{\theta_0+cv - \theta}} \in \C$.
    Taking $c \to \infty$ and using the fact that $\C$ is closed yields $\theta + \frac{v}{\norm{v}} \in \C$ and thus $v \in T_\C(\theta)$. Since $\theta$ was arbitrary, we have $v \in K_\C$.

    Conversely, suppose $v \in K_\C$. Let $c^* \coloneqq \sup\{c>0 : \theta_0 + cv \in \C\}$.
    The supremum is over a nonempty set because $v \in T_\C(\theta_0)$.
    Suppose for sake of contradiction that $c^* < \infty$. Since $\C$ is closed, $\theta_0+c^* v \in \C$. Thus $v \in T_\C(\theta_0+c^* v)$ which implies $\theta_0+(c^*+\alpha) v \in \C$ for some $\alpha>0$, contradicting the definition of $c^*$. Thus $c^*=\infty$ and $\theta_0+cv \in \C$ for all $c>0$.

    \item
    Both sides can be expressed as the set of $v \in \R^n$ satisfying $\theta_0 + \sigma v \in \C$ for all $\sigma>0$.
\end{enumerate}

We now prove the second part of the lemma.
The definition \eqref{equation:def_core_cone} implies $K_\C \subseteq T_\C(\theta)$ for any $\theta \in \C$.

Now, assume $F_\C(\theta_0)$ is a cone.
If the reverse inclusion $T_\C(\theta) \subseteq F_\C(\theta)$ holds, then $\theta_0 - \theta \in T_\C(\theta) = F_\C(\theta)$ so $\theta_0 - (\theta-\theta_0) \in \C$.
Conversely, suppose $\theta_0 - (\theta-\theta_0) \in \C$.
If $v \in T_\C(\theta)$, then $\theta + cv \in \C$ for some $c>0$. By convexity,
$\theta_0 + cv/2 \in \C$, so $v \in F_\C(\theta_0)$. Thus $T_\C(\theta) \subseteq F_\C(\theta)$.
\end{proof}

\begin{proof}[Proof of \autoref{proposition:high_sigma_limit}]
We use $Y$ instead of $\theta^* + \sigma Z$ throughout the proof,
but note that $Y$ depends on $\sigma$.

Without loss of generality we can translate the problem so that $\Pi_\C(\theta^*)=0$.

In view of \eqref{equation:dom_conv}, we may use the dominated convergence theorem on $\sigma^{-2} \norm{\Pi_\C(Y) - \Pi_\C(\theta^*)}^2$, so
\begin{align}
&\lim_{\sigma \to \infty}
\frac{1}{\sigma^2} \E\norm{\Pi_\C(Y) - \Pi_\C(\theta^*)}^2
\\
&= \E \lim_{\sigma \to \infty}\frac{1}{\sigma^2}
\norm{\Pi_\mathcal{C}(Y) - \Pi_\C(\theta^*)}^2
& \text{dom. conv. with $\E \norm{Z}^2$}
\\
&= \E \lim_{\sigma \to \infty}\frac{1}{\sigma^2}
\norm{\Pi_\mathcal{C}(Y)}^2
\\
&\overset{(i)}{=} \E \norm{\Pi_{K_\C}(Z)}^2 = \delta(K_\C),
\end{align}
where we verify the equality (i) below.

Similarly, \eqref{equation:dom_conv} allows us to use the dominated convergence theorem again for the excess risk.
\begin{align}
&\lim_{\sigma\to \infty}
\frac{1}{\sigma^2} \parens*{
    \E \norm{\Pi_\C(Y) - \theta^*}^2
    - \norm{\Pi_\C(\theta^*) - \theta^*}^2
}\\
&= \E \lim_{\sigma \to \infty} \frac{1}{\sigma^2} \parens*{
    \norm{\Pi_\C(Y) - \theta^*}^2
    - \norm{\theta^*}
}
& \text{dom. conv. with $\E \norm{Z}^2$}
\\
&= \E \lim_{\sigma \to \infty} \frac{1}{\sigma^2}
\parens*{
    \norm{\Pi_\C(Y)}^2
    - 2 \inner{\Pi_\C(Y), \theta^*}
}
\\
&\overset{(ii)}{=} \E \norm{\Pi_{K_\C}(Z)}^2 = \delta(K_\C).
\end{align}

It remains to verify (i) and (ii).

\begin{enumerate}[(i)]
    \item
    \begin{align}
    &\phantom{{}\le{}}
    \abs*{
        \frac{1}{\sigma^2} \norm{\Pi_\C(Y)}^2
        - \norm{\Pi_{K_\C}(Z)}^2
    }
    \\
    &\le \frac{1}{\sigma^2} \abs*{
        \norm{\Pi_\C(Y)}^2
        - \norm{\Pi_{K_\C}(Y)}^2
    }
    + \abs*{
        \frac{1}{\sigma^2} \norm{\Pi_{K_\C}(Y)}^2
        - \norm{\Pi_{K_\C}(Z)}^2
    }
    \\
    &\le \frac{c}{\sigma^2}
    + \abs*{
        \norm{\Pi_{K_\C}(\theta^*/\sigma + Z)}^2
        - \norm{\Pi_{K_\C}(Z)}^2
    }
    & \text{\autoref{lemma:boundedness_nonnegative}; $K_\C$ is a cone}
    \\
    &\overset{\sigma \to \infty}{\longrightarrow} 0.
    & \text{$x \mapsto \norm{\Pi_{K_\C}(x)}^2$ is continuous}
    \end{align}

    \item We already showed $\norm{\Pi_\C(Y)}^2/\sigma^2 \to \norm{\Pi_{K_\C}(Z)}^2$, so it suffices to show the cross term vanishes.
    Indeed, we have $\norm{\Pi_\C(Y)}/\sigma \to \norm{\Pi_{K_\C}(Z)}$, so
    \begin{equation}
    \frac{1}{\sigma^2}
    \abs*{
        \inner{\Pi_\C(Y),\theta^*}
    }
    \le \frac{1}{\sigma^2} \norm{\Pi_\C(Y)} \norm{\theta^*}
    \overset{\sigma\to \infty}{\longrightarrow} 0.
    \end{equation}
\end{enumerate}
\end{proof}

\begin{proof}[Proof of \autoref{corollary:high_sigma_limit_ex}]

We begin with the first claim.
Since $\C=\R^n_+$ is a cone, we have $K_\C=\R^n_+$.
Provided we verify \eqref{equation:boundedness}, the result follows from \autoref{proposition:high_sigma_limit}.
Let $\theta \coloneqq \Pi_\C(\theta^*)$ and fix $x \in \R^n$.
Then some casework yields
\begin{equation}
\norm{\Pi_{F_\C(\theta)}(x)}^2
- \norm{\Pi_{K_\C}(x)}^2
= \sum_{i=1}^n \max\{x_i,-\theta_i\}^2
- \sum_{i=1}^n \max\{x_i,0\}^2
\le \sum_{i=1}^n \theta_i^2
= \norm{\theta}^2
\eqqcolon c.
\end{equation}

We now turn to the second claim.
If $\C$ is bounded, then by \autoref{lemma:core_cone}, $K_\C=\{u : \R_+ u \subseteq F_\C(\theta_0)\}=\{0\}$ for any $\theta_0 \in \C$.

Conversely, suppose $\C$ is unbounded and fix $\theta_0 \in \C$.
Let
\begin{equation}
U_r \coloneqq \{v \in S^{n-1} : \theta_0 + cv \notin \C \text{ for some } c \in (0,r)\}.
\end{equation}
This set is open: if $(v_n)$ is a sequence in $U_r^c$ converging to $v$,
then the fact that $\C$ is closed implies $\theta_0 + r v_n \in \C$
for all $n$,
and consequently $\theta_0 + rv \in \C$ and finally $v\in U^c_r$.

If $\bigcup_{r > 0} U_r$ is an open cover of the compact set $S^{n-1}$,
then $S^{n-1} \subseteq U_r$ for some $r>0$, which implies $\C \subseteq B_r(\theta_0)$,
a contradiction.
Thus, some direction $v \in S^{n-1}$ does not lie in $\bigcup_{r>0} U_r$,
i.e., $\theta_0 + cv \in \C$ for all $c \ge 0$. This implies $cv \in K_\C$ for all $c \ge 0$.

We now apply \autoref{proposition:high_sigma_limit}.
If $\C$ is bounded, then so is $F_\C(\Pi_\C(\theta^*))$. Choosing $c$ large enough so that $F_\C(\Pi_\C(\theta^*))$ lies in the ball of radius $c$ suffices to satisfy \eqref{equation:boundedness}.
Then \autoref{proposition:high_sigma_limit} implies that the high $\sigma$ limits are $\delta(K_\C)=0$.
\end{proof}

\section*{Acknowledgments}
We are thankful to Dennis Amelunxen for an informative email
correspondence and to Bodhisattva Sen for helpful discussions.

\bibliographystyle{chicago}
\bibliography{AG}

\end{document}